\documentclass[a4paper,11pt,reqno]{amsart}  \usepackage{epsfig}
\usepackage{color}      
\usepackage{amssymb,amsmath,amsthm,amstext,amsfonts}
\usepackage{graphicx, color} 
\usepackage{amsmath,amsthm,amsfonts,amssymb,amscd}
\usepackage[T1]{fontenc}
\usepackage[mathscr]{eucal}
\usepackage{dsfont}

\usepackage{psfrag}
\usepackage{url}
\usepackage{epstopdf}
\usepackage{epic,eepic}
\usepackage{upgreek}
\usepackage{amssymb}
\usepackage{mathrsfs}
\usepackage{verbatim}
\usepackage{mathrsfs}
\usepackage{amstext}
\usepackage{amsthm}
\usepackage{amssymb}
\usepackage{graphicx}

\usepackage[colorlinks=true, linkcolor=blue, urlcolor=red, citecolor=blue, hyperindex]{hyperref}

\makeatletter \@addtoreset{equation}{section} \makeatother

\renewcommand\thetable{\thesection.\@arabic\c@table}

\theoremstyle{plain}
\newtheorem{maintheorem}{Theorem}

\newtheorem{theorem}{Theorem}[section]
\newtheorem{proposition}{Proposition}[section]
\newtheorem{lemma}{Lemma}[section]
\newtheorem{corollary}{Corollary}[section]
\newtheorem{definition}{Definition}[section]
\newtheorem{remark}{Remark}[section]
\newtheorem{example}{Example}[section]

\newcommand{\de} {\delta}       
\newcommand{\vep}{\varepsilon}

\newcommand{\si} {\sigma}

\newcommand{\N}{\mathbb{N}}
\newcommand{\R}{\mathbb{R}}
\newcommand{\supp}{\operatorname{supp}}
\newcommand{\diam}{\operatorname{diam}}

\newcommand{\cE}{\mathcal{E}}

\newcommand{\cG}{\mathcal{G}}

\newcommand{\cQ}{\mathcal{Q}}

\newcounter{main}

\title[Weak Gibbs Measures for Local Homeomorphisms]
{Weak Gibbs Measures for Local Homeomorphisms}

\author{Giovane Ferreira and Vanessa Ramos}
\address{Giovane Ferreira, Departamento de Matem\'atica, Universidade Federal do Maranh\~ao\\
Av. dos Portugueses, 1966, Bacanga, 65065-545, São Luís, BRAZIL.}
\email{giovane.ferreira@ufma.br}

\address{Vanessa Ramos, Departamento de Matem\'atica, Universidade Federal do Maranh\~ao\\
	Av. dos Portugueses, 1966, Bacanga, 65065-545, São Luís, BRAZIL.}
\email{ramos.vanessa@ufma.br}

\thanks{GF was supported by PRAPG-CAPES (Brazil).} 
\thanks{VR was supported by PRAPG-CAPES (Brazil) and by Universal-CNPQ (Brazil).}

\begin{document}

\begin{abstract} 
We study a broad class of local homeomorphisms and continuous potentials, proving the existence and uniqueness of weak Gibbs measures. From the Gibbs property, we show the uniqueness of equilibrium states and derive a large deviations principle. Furthermore, we extend these results to a class of attractors that are semiconjugate to local homeomorphisms from our original setting. Our approach is based on non-uniform conformal-like property and applies to a wide range of topological dynamical systems, including non-uniformly expanding maps, zooming local homeomorphisms, attractors arising from solenoid-like constructions and certain families of partially hyperbolic horseshoes.
\end{abstract}

%\keywords{confomal measure, large deviations.}
% \footnotetext{2010 {\it Mathematics Subject classification}:
%Primary 
%37C40, % Smooth ergodic theory, invariant measures
%37D25; % non-uniformly hyperbolic systems (Lyapunov exponents, Pesin theory, etc.)
%Secondary 
%37D30, %Partially hyperbolic systems and dominated splittings 
%37D35. %Thermodynamic formalism, variational principles, equilibrium states 
%37F15,  % Expanding maps; hyperbolicity; structural stability
%37D20, %	Uniformly hyperbolic systems (expanding, Anosov, Axiom A, etc.)}

\maketitle

%\tableofcontents

\date{\today}

%%%%%%%%%%%%%%%%%%%%%%%%%%%%
\section{Introduction}
The interplay between statistical physics and dynamical systems has produced a powerful framework for studying invariant measures, statistical properties, and dimensional characteristics of complex systems. Central to this theory are {\textit{Gibbs measures}}, which exhibit strong regularity properties and provide precise exponential estimates for the measure of dynamical balls in terms of the pressure and the Birkhoff sum of the potential.

The concept of Gibbs measures was introduced by Yakov G. Sinai in his pioneering work~\cite{Sinai}, where he transferred ideas from statistical mechanics into the realm of dynamical systems - particularly to Anosov diffeomorphisms. His construction of Gibbs measures laid the groundwork for defining thermodynamic quantities in dynamical systems, enabling the rigorous study of entropy, pressure, equilibrium states, SRB measures, statistical stability and multifractal analysis.

Building on Sinai’s foundation,  Bowen~\cite{Bowen} and Ruelle~\cite{Ruelle68}, \cite{Ruelle78} further developed the theory, extending it to uniformly hyperbolic (Axiom A) systems. Their contributions helped establish thermodynamic formalism as one of the cornerstones of modern dynamical systems theory.

Despite its success in uniformly hyperbolic settings, the classical notion of a Gibbs measure often fails to extend to systems lacking uniform expansion, Markov partitions, or regular potentials. In such contexts, the strict exponential bounds required by the Gibbs property may no longer hold. This limitation has motivated the development of weakened versions of the Gibbs condition - most notably the concept of {\textit {weak Gibbs measures}} - which relax the uniformity of the constants or the time scales involved, while still preserving a form of exponential control over the measure of dynamical balls. Despite this relaxation, weak Gibbs measures retain enough structure to characterize equilibrium states and to support the study of statistical properties of the system. Understanding such measures is essential for extending thermodynamic formalism to broader classes of dynamical systems, where the classical Gibbs property may fail or only hold in a weaker sense.

The existence of weak Gibbs measures has been  established in systems exhibiting non-uniform expansion~\cite{OV08, VV10}, weaker forms of hyperbolicity \cite{Yur00}, 
% partially hyperbolic dynamics \cite{CV17}, 
irreducible shift  spaces \cite{PS19}, irregular or low-regularity potentials \cite{Yur03}, asymptotically additive sequences of potentials \cite{IommiYayama17} and other structures. In a broader context, we study a general class of local homeomorphisms and continuous potentials, and we prove the existence and uniqueness of an ergodic weak Gibbs measure under a flexible condition inspired by a non-uniform conformal-like property. We also extend this result to a class of attractors semiconjugated to local homeomorphisms within our original setting. Our approach applies to a wide range of topological dynamical systems, including non-uniformly expanding maps~\cite{OV08, RV17, VV10}, zooming local homeomorphisms~\cite{Pin11}, attractors arising from solenoid-like constructions~\cite{ARS18, CN, FO}, and certain families of partially hyperbolic horseshoes~\cite{diazetal, RS17, RS16}. 

The strategy for extending our results from local homeomorphisms to attractors semiconjugated to them consists in proving that the Gibbs property is preserved under projection and liftability. This problem, in the symbolic context, has attracted considerable attention, see~\cite{CU03,CU11,FO18,Kemp11,PK11, Verb10, Yaya16, Yoo10}.  

% Important findings address the following problem: given a Gibbs measure $\mu$ on subshift spaces $\Sigma_1 = \{1,\dots,k_i\}^\mathbb{N}$ for a continuous function $\psi_1$,  consider its image $\nu:=\Pi_*\mu$ under  one block factor map $\Pi$. An interesting question that arises in the theory of Hidden Markov Chains is to show that $\nu$ is also a Gibbs measure for some continuous function $\psi_2$. 

An interesting application of our main results is that the weak Gibbs measures constructed in these settings minimizes the free energy of the system. More precisely, let $T: M \to M$ be a continuous map  defined on a compact metric space $M$ and let $\phi: M\to \mathbb{R}$ be a continuous function. Denoting the set of $T$-invariant probability measures by $\mathbb P_T(X)$, we say that $\mu_{\phi}\in \mathbb P_T(X)$ is an \textit{equilibrium state} for $(T,\phi)$ if it satisfies the variational principle
$$h_{\mu_{\phi}}(T)+\int{\phi}\, d\mu_{\phi}=\sup_{\eta\in\mathbb P_T(X)} \left\{ h_{\eta} (T) + \int \phi \, d\eta\right\},$$ where $h_\eta(T)$ is the metric entropy of $\eta$. We refer the reader to \cite{W} for a proof of the variational principle.

A fundamental question that naturally arise in the thermodynamic formalism is: {\textit {When is an equilibrium state a Gibbs measure?}} In the uniformly hyperbolic setting, the answer is well understood - the equilibrium state is unique and satisfies the classical Gibbs property, see~\cite{Bow75}. However, beyond the hyperbolic framework, particularly in systems with non-uniform expansion or partial hyperbolicity, the situation becomes significantly more subtle. Despite notable advances in recent years~\cite{ARS18, CN, CT, FO, RV17, RS17, RS16}, the characterization and existence of Gibbs-type equilibrium states in such settings remain challenging and far from complete.

In this work, we address this problem by proving that the weak Gibbs measure obtained through our construction - both for local homeomorphisms and for attractors semiconjugated to them - is in fact the unique equilibrium state for the system. This result provides a generalization of the classical theory beyond the class of non-uniformly hyperbolic dynamical systems. Furthermore, exploiting the weak Gibbs property, we also establish a large deviations principle for the equilibrium state.

The study of the Large Deviations Principle (LDP) plays a central role in understanding the fluctuation behavior of dynamical systems beyond the typical or average behavior described by ergodic theorems. It provides quantitative estimates for the exponential rate at which the probabilities of deviations from expected values decay. In the classical thermodynamic formalism, the LDP is closely linked to the Gibbs property, as Gibbs measures provide exponential control over the  dynamical balls in terms of the Birkhoff sum of the potential.

Since the nineties, the theory of large deviations has attracted significant attention from many authors. In~\cite{Ki90, KiNew90, You90}, Kifer, Newhouse, and Young established large deviation principles for equilibrium states associated to H\"older continuous potentials in the setting of uniformly hyperbolic transformations.  In~\cite{MN08}, Melbourne and Nicol obtained large deviation estimates for non-uniformly hyperbolic systems. For shift with countably many symbols, large deviation results were proved by Yuri (see~\cite{Yuri}).
In the setting of non-uniformly expanding maps, Ara\'ujo and Pac\'ifico~\cite{AP06} obtained upper bounds for the deviation sets of physical measures. In~\cite{Var12}, Varandas derived large deviation estimates for weak Gibbs measures associated to non-uniformly expanding maps satisfying the non-uniform specification property. Later, Bomfim and Varandas~\cite{BV19} introduced the gluing orbit property, a condition weaker than non-uniform specification, and proved large deviation results for semiflows satisfying this property.
In ~\cite{CFV19}, Cruz, Ferreira, and Varandas obtained large deviation upper bounds for the SRB measure associated to a class of partially hyperbolic endomorphisms.

This work also contributes to the theory of large deviations in a topological setting that requires no form of expansion - not even in a non-uniform sense - thus offering an extension of large deviation principles to dynamical systems beyond the realm of hyperbolicity.
Building upon the techniques developed by Young~\cite{You90}, we establish upper bounds for large deviations with respect to deviation sets associated with weak Gibbs measures constructed within our framework. Additionally, we derive large deviation lower bounds over the set of ergodic measures. Using the non-uniform gluing property introduced in~\cite{BV19}, we further extend these lower bounds to hold over the full set of invariant measures.

This paper is organized as follows. In Section~\ref{results}, we define a weak Gibbs measure and state our main results. Section~\ref{exemplos} presents some examples within our setting. In Section~\ref{toppressure}, we recall some properties of the relative topological pressure for non-compact sets. Section~\ref{reference measure} is devoted to the construction of a (non-invariant) reference measure that satisfies a weak Gibbs property. The existence and uniqueness of an ergodic weak Gibbs measure for local homeomorphisms satisfying our assumptions are established in Section~\ref{weak}. In Section~\ref{atrator}, we address the problem of liftability and projection of weak Gibbs measures. Finally, in Section~\ref{LD}, we establish a large deviation principle for the weak Gibbs measures constructed in our framework.
%%%%%%%%%
\section{Statement of results}\label{results}

%%%%%%%%%%%%%%%%%%%%%%%%%%%%%%

Consider $M$ a compact metric space. Let $f : M \rightarrow M$ be a local homeomorphism with dense preorbit $\{f^{-j}(x)\}_{j\geq0}$, for every $x\in M$. Suppose the existence of a subset $\cG\subset M$ where the points have times in which the dynamics sends the dynamical ball homeomorphically onto the open ball, i.e., for some $\delta_0>0$ and each $x\in \cG$ there exist a sequence $(n_j(x))_j$ of natural numbers such that
\begin{equation*}\tag{*} \label{estrela}
f^{n_j(x)}(B_{\varepsilon}(x,n_j(x)))=B(f^{n_j(x)}(x),\vep)
\end{equation*}
homeomorphically for all $j\in\mathbb{N}$ and $0<\varepsilon<\delta_0,$ where $B_{\varepsilon}(x,n_j(x))$ is the dynamical ball of $x$ and $B(f^{n_j(x)}(x),\vep)$ is the open ball around $f^{n_j(x)}(x)$; for the definition of dynamical ball, see Section~\ref{toppressure}. 
 	
%\begin{remark} \marginpar{Verdade?}
%If $n$ is a time with the property (\ref{estrela}) for $x$ then $n-l$ is a time with the property (\ref{estrela}) for $f^l(x)$, for any $1\le l<n$. 
%\end{remark}

We also assume that the points of $\cG$ have positive frequency of times with the property (\ref{estrela}) i.e. there exists $\theta>0$, that depends only on $f$, such that
\begin{equation*}\tag{H1}\label{H1}
\cG:=\left\{x \in M: \limsup_{n\rightarrow +\infty}\frac{1}{n}\#\{1\leq j\leq n\, ;\, f^{j(x)}\,\, \mbox{satisfies (\ref{estrela})}\}\geq \theta \right\}.
\end{equation*}
Notice that $\cG$ is a positively invariant set ($f(\cG)\subset \cG$). 

Let $\varphi:M\rightarrow \R$ be a continuous function. Suppose that $\varphi$ satisfies the Bowen property on the set $\cG$: there exist $V>0$ and $\varepsilon>0$ such that 
\begin{equation*}\tag{H2}\label{H2}
\sup_{y, z\in B_\vep(x,n_{j}(x))}\left\{|\sum_{i=0}^{n_{j}(x)}\varphi(f^i(y))-\varphi(f^i(z))|\right\}\leq V.
\end{equation*}
for every $x\in \cG$ and $n_{j}(x)$ satisfying~(\ref{estrela}). 

We also assume that the topological pressure of $\varphi$ is equal to the relative pressure of $\varphi$ on the set $\cG$ i.e. 
\begin{equation*}\tag{H3}\label{H3}
P_{f}(\varphi, \cG^{c})<P_{f}(\varphi, \cG)=P_{f}(\varphi).
\end{equation*}
For the definition of topological pressure relative to a set, see Section~\ref{toppressure}. In Proposition~\ref{pr.openess}, we will show that the set of functions satisfying Condition~(\ref{H3}) is open in the $C^0$- topology. 

Let $\mathbb{P}(M)$ be the space of probability measures defined in the Borel sets of $M.$ We say that $\mu\in\mathbb{P}(M)$ is a \emph{weak Gibbs measure} for $(f, \varphi)$ if for some constants $P \in \R$ and $\delta> 0$ the following holds: for all $0<\varepsilon\leq\delta$ there exists $K=K(\varepsilon)>0$ so that for $\mu$-almost every $x\in M$ there exists a sequence $(n_k)_k=(n_k(x))_k$ such that
\begin{equation*}\tag{**} \label{2estrela}
  K^{-1} \leq  \frac{\mu(B_\varepsilon(x, n_{k}))}{\exp(S_{n_{k}}\varphi(y) - n_{k}P)}  \leq K,
	\end{equation*}
for all $y\in B_\varepsilon(x, n_k)$ and every $k\geq 1$.

\begin{remark}\label{property hyp times}	The sequence $(n_k(x))_k$ where the measure $\mu$ has the Gibbs property~(\ref{2estrela}) is called the Gibbs times of $x$. When the sequence $(n_k(x))_k$ coincides with the set of natural numbers $\mathbb{N}$ then $\mu$ is called a Gibbs measure in the sense of Bowen~\cite{Bow75}. In Section~\ref{reference measure} we will show that any time of $x$ satisfying~(\ref{estrela}) is a Gibbs time for $x\in\cG.$
\end{remark} 

Note that the definition of a weak Gibbs measure depends on the constant $P$. In our main result, we establish the existence of a unique ergodic invariant weak Gibbs measure for $(f, \varphi)$, and we show that the constant $P$ is determined by the relative pressure of the set $\cG.$

\begin{maintheorem}\label{Gibbs base} 
Consider $f:M\to M$ a local homeomorphism satisfying~(\ref{H1}) and ${\varphi}: M \to \mathbb{R}$ a continuous potential satisfying~(\ref{H2}) and (\ref{H3}). Then there exists a unique ergodic weak Gibbs measure $\mu_{\varphi}$ for $(f, {\varphi})$. Moreover, the relative topological pressure of $\cG$ is equal to $P$.
\end{maintheorem}

As a first application of our main result, we show that the weak Gibbs measure constructed in our setting is the unique equilibrium state of the system. The problem of existence and uniqueness of equilibrium states for certain classes of non-uniformly expanding maps and H\"older continuous potentials has been studied extensively; see, for instance,~\cite{OV08, RV17, RS16, VV10}. In our second result, we obtain the uniqueness of the equilibrium state in a broader setting that includes a wider class of local homeomorphisms and continuous potentials, thereby extending the scope of the aforementioned results.

\begin{maintheorem}\label{ee base}
Let $f:M\to M$ be a local homeomorphism satisfying~(\ref{H1}) and let ${\varphi}: M \to \mathbb{R}$ be a continuous potential satisfying~(\ref{H2}) and (\ref{H3}). Suppose that $\cG$ admits a generating partition. Then the unique ergodic weak Gibbs measure $\mu_{\varphi}$ is the unique equilibrium state of $(f, \varphi).$
\end{maintheorem}

In this setting, we also establish a large deviation principle for the weak Gibbs measure $\mu_{\varphi}$, assuming that the sequence of Gibbs times $(n_k(x))_k$ is non-lacunar; that is, it satisfies $\displaystyle\lim_{k\rightarrow +\infty}\frac{n_{k+1}(x)-n_k(x)}{n_k(x)}=0,$ for every $x\in \cG.$
In particular, when the sequence of Gibbs times $(n_k(x))_k$ is syndetic in $\cG$, meaning there exists  $C>0$ such that $n_{k+1}(x)-n_k(x)\le C$ for every $x\in \cG$ and $k\in\mathbb{N}$, the lower bound for the deviation sets can be estimated over the set of invariant measures that give full mass to $\cG$ . 

\begin{maintheorem}\label{th ld1}
Let $\mu_{\varphi}$ be the weak Gibbs measure of $(f,\varphi)$ as in Theorem~\ref{Gibbs base}. Suppose that the sequence of Gibbs times of $x\in\cG$ is non-lacunar. For any continuous function $\psi:M \rightarrow \R$ and $c\in \R$, we have for every 
$\beta > 0$ small the following upper bound
$$
\limsup_{n \to +\infty} \frac{1}{n} \log \mu_{\varphi}\left(\{x \in M: \frac{1}{n}\sum_{j=0}^{n-1}\psi(f^j(x))\geq c\}\right)
\leq 
\max \{\cE(\beta) , 
I(c)+\beta\},
$$
where $$
{\mathcal E}(\beta)=\limsup_{n\to+\infty} \frac1n \log \mu_{\varphi}\Big(x\in \cG: n_{i+1}(x)-n_i(x)>\frac{\beta n}{2(\sup|\varphi|+P_f(\varphi))} \Big)
$$
for $n_i(x), n_{i+1}(x)$ consecutive Gibbs times of $x\in\mathcal{G}$ such that $n_{i}\leq n\leq n_{i+1}$
and 
$$
I(c)= \sup \{ h_\eta(f) + \int \varphi d\eta -P_f(\varphi)\},
$$
with the supremum taken over all $f$-invariant probability measures $\eta$ such that $\eta(\cG)=1$ and $\int \psi d \eta>c$.
Moreover, it holds the lower bound
$$
\liminf_{n\rightarrow +\infty}\frac{1}{n}\log \mu_\varphi\left(\{x \in M: \frac{1}{n}\sum_{j=0}^{n-1}\psi(f^j(x))>c\} \right)\ge h_{\eta}(f)+\!\int\varphi d\eta -P_f(\varphi),
$$
for any ergodic probability measure $\eta$ such that $\eta(\cG)=1$ and $\int \psi d \eta>c$. 

In addition, if the sequence $(n_k(x))_k$ is syndetic in $\cG$, then the lower bound stated above can be taken over all invariant measures $\eta$ satisfying $\eta(\cG)=1$ and $\int \psi d \eta>c$. 
\end{maintheorem} 

The first large deviation estimate provides an upper bound on the speed of convergence of Birkhoff averages, characterized as the maximum of two exponential decay rates: one corresponding to the set of times ${\mathcal E}(\beta)$ where the Gibbsian property holds, and the other defined via a rate function $I(\cdot)$ which quantifies the deviation of measures to the equilibrium.

%%%%%%%%%%%%%%%%%%%%%%%%%%%%%%%%%%%%%%
\subsubsection*{\bf Attracting Dynamics}\label{section attractor}
In the second part of this work, we study homeomorphisms semiconjugated to the dynamics of the first setting. As before, consider $f: M\to M$ a local homeomorphism with dense preorbit $\{f^{-j}(x)\}_{j\geq0}$ for every $x\in M$ satisfying condition~(\ref{H1}). Let $N$ be a compact metric space and let $F:N\to N$ be a homeomorphism onto its image. Suppose the existence of a continuous surjection $\pi: N\to M$ which conjugates $f$ and $F$, that is, $$\pi\circ F=f\circ \pi. $$ 
For $x\in M, $ we set $N_x=\pi^{-1}(x).$ Notice that $N_x$ is compact, $N=\bigcup_{x\in M}{N_x}$ and $F(N_x)\subset N_{f(x)}$. Assume that $N_x$ are local stable manifolds: there exists $0<\lambda<1$ such that 
$$d_N(F(\hat{y}), F(\hat{z}))\leq \lambda d_N(\hat{y}, \hat{z}),$$
for every $\hat{y}, \hat{z}\in N_x$ and all $x\in M,$ where $F_x$ is the restriction $F|_{N_x}$ and $d_N$ is the metric on $N$. 

%In this context, we consider the attractor set $\Lambda=\bigcap_{n=0}^{+\infty} F^n(N)$ which is compact and invariant by $F.$ Moreover, any invariant measure is supported in this set.

Finally, we suppose the existence of $F$-invariant holonomies, that is,  given $\hat{x}, \hat{y} \in N$ for $x = \pi(\hat{x}), y=\pi(\hat{y})$ in $M$,  there exist holonomies $h_{x,y}:  N_x \to N_y$ and a constant $C\geq 1$ such that
$$\frac{1}{C}[d_M (x, y) + d_N(h_{x,y}(\hat{x}), \hat{y})]\leq d_N(\hat{x}, \hat{y}) \leq  C [d_M (x, y) + d_N(h_{x,y}(\hat{x}), \hat{y})],$$
where $d_M$ and $d_N$ are the metrics on $M$ and $N$ respectively. The holonomies are invariant by $F$ that is $$F(h_{x, y}(z))=
h_{f(x),f(y)}F(z).$$

Let $\cG\subset M$ as described in (\ref{H1}) and consider the set $\hat{\cG}:=\pi^{-1}(\cG)\subset N.$ For a continuous potential $\phi:N\rightarrow \R$ we suppose the Bowen property on the set $\hat{\cG}$: there exist $\varepsilon>0$ and $V>0$ such that 
\begin{equation*}\tag{H4}\label{H4}
\sup_{\hat{y}, \hat{z}\in B_\vep(\hat{x},n)}\left\{|\sum_{j=0}^{n}\phi(F^j(\hat{y}))-\phi(F^j(\hat{z}))|\right\}\leq V.
\end{equation*}
for every $\hat{x}\in \hat{\cG}$ and $n\in \mathbb{N}.$ Moreover, we assume that the topological pressure of $\phi$ is located on $\hat{\cG}$ i.e. 
\begin{equation*}\tag{H5}\label{H5}
P_{F}(\phi, \hat{\cG}^{c})<P_{F}(\phi, \hat{\cG})=P_{F}(\phi).
\end{equation*}

Addressing the problem of whether the Gibbs property is preserved under projection and liftability, we prove the existence of a unique ergodic weak Gibbs measure for the attractor. This measure is, in fact, obtained as the projection of the weak Gibbs measure provided by Theorem~\ref{Gibbs base}.

%%%%%%%%%%%%%%%%%%%%%%%%%

\begin{maintheorem}\label{theo.Gibbs} 
Let $F:N\to N$ be as described above and let ${\phi}: N \to \mathbb{R}$ be a H\"older continuous potential satisfying~(\ref{H4}) and (\ref{H5}). Then there exists a unique ergodic weak Gibbs measure $\mu_{\phi}$ associated to $(F, {\phi})$. 
%Moreover, the relative topological pressure of the set $\hat{\cG}$ is equal to $P$.
\end{maintheorem}  

In previous works, Alves, Ramos and Siqueira~\cite{ARS18}, Castro and Nascimento~\cite{CN}, Fisher and Oliveira~\cite {FO} studied attractors semiconjugated to certain classes of non-uniformly expanding maps and proved the uniqueness of equilibrium states. In this work, we generalize these results to the broader setting of attractors semiconjugated to local homeomorphisms satisfying~(\ref{H1}). Furthermore, we establish the Gibbs property and obtain a large deviations principle.

\begin{maintheorem}\label{theo finitude}
Let $F:N\to N$ be as described above and let ${\phi}: N \to \mathbb{R}$ be a H\"older continuous potential satisfying~(\ref{H4}) and (\ref{H5}). If the set $\cG=\pi(\hat{\cG})$  has  a generating partition then the unique ergodic weak Gibbs measure $\mu_{\phi}$ is the unique equilibrium state of $(F, \phi)$.
\end{maintheorem}

Finally, from the Gibbs property of the equilibrium state $\mu_{\phi}$, we derive a large deviations principle for the attractor. 

\begin{maintheorem}\label{th ld2}
Let $\mu_{\phi}$ be the weak Gibbs measure of $(F,\phi)$ as in Theorem~\ref{theo.Gibbs}. If the sequence of Gibbs times of $y\in\hat\cG=\pi^{-1}(\cG)$ is non-lacunar then for any continuous observable $\psi:M \rightarrow \R$ and $c\in \R$, we have for every 
$\beta > 0$ small the following upper bound
$$
\limsup_{n \to +\infty} \frac{1}{n} \log \mu_{\phi}\left(\{y \in N: \frac{1}{n}\sum_{j=0}^{n-1}\psi(F^j(y))\geq c\}\right)
\leq 
\max \{\cE(\beta) , 
I(c)+\beta\},
$$
where 
$$
{\mathcal E}(\beta)=\limsup_{n\to+\infty} \frac1n \log \mu_{\phi}\Big(y\in\hat{\cG}: n_{i+1}(y)-n_i(y)>\frac{\beta n}{2(\sup|\phi|+P_F(\phi))} \Big)
$$
for $n_i(y), n_{i+1}(y)$ consecutive Gibbs times of $y\in\hat{\mathcal{G}}$ such that $n_{i}\leq n\leq n_{i+1}$
and 
$$
I(c)= \sup \{ h_\eta(F) + \int \phi d\eta -P_{\phi}(F)\}
$$
with supremum taken over all  $F$-invariant probability measures $\eta$ such that $\eta(\hat\cG)=1$ and $\int \psi d \eta>c$.
Moreover, it holds the lower bound 
$$
\liminf_{n\rightarrow +\infty}\frac{1}{n}\log \mu_\phi\left(y \in N: \frac{1}{n}\sum_{j=0}^{n-1}\psi(F^j(x)>c \right)\ge h_{\eta}(F)+\!\int\varphi d\eta -P_F(\phi),
$$
for any ergodic probability measure $\eta$ such that $\eta(\hat\cG)=1$ and $\int \psi d \eta>c$.

In addition, if the sequence $(n_k(y))
_k$ is syndetic in $\hat\cG$, then the lower bound
stated above can be taken over all invariant measures $\eta$ satisfying $\eta(\hat\cG) = 1$ and $\int \psi d\eta > c.$

% $$
% \hat{{\mathcal E}}_{\mu_{\phi}}(\beta)=\limsup_{n\to\infty} \frac1n \log \mu_{\phi}\Big(\hat{y} \colon \hat{y}\in \pi^{-1}(x) \;\mbox{and }\;\exists i\; \mbox{\textrm{satisfying}}\; n_{i+1}(x)-n_i( x)>\beta n_i \Big)
% $$

\end{maintheorem}

The approach to extending our results from local homeomorphisms to attractors semiconjugated to them relies on showing that the Gibbs property is preserved under projection and liftability. In addition, to establish the large deviation principle, we prove that the non-uniform gluing property holds for both systems.

%%%%%%%%%%%%%%%%%%%%%%%%%%%%%%%%%

\section{Examples}\label{exemplos} 

We describe some examples of systems which satisfy the requirements of our results. We start by presenting a robust class of non-uniformly expanding maps introduced by Alves, Bonatti and Viana \cite{ABV00} and studied by  Oliveira and Viana \cite{OV08} and Varandas and Viana \cite{VV10}. 

\begin{example}\label{paulo} \normalfont{Let $f:M\rightarrow M$ a be $C^{1}$ local diffeomorphism defined on a compact manifold $M$. For  $\delta>0$ small and $\sigma<1$, consider  a covering $\mathcal Q=\left\{Q_{1},\dots, Q_{q}, Q_{q+1},\dots, Q_{s}\right\}$ of $M$ by domains of injectivity for $f$ and a  region $\textsl{A}\subset M$ satisfying:
		\begin{enumerate}
			\item[(1)]   $\|Df^{-1}(x)\|\leq 1+\delta$, for every $x\in\textsl{A}$;
			\item[(2)] $\|Df^{-1}(x)\|\leq \sigma $, for every $x\in M\setminus\textsl{A}$;
			\item[(3)]   $A$   can be covered by $q$ elements of the partition $\mathcal Q$ with $q<\deg(f).$
		\end{enumerate}
		In~\cite{OV08, VV10} it was showed the existence of a non-uniformly expanding set $H\subset M$ where the points $x\in H$ satisfy
		\begin{equation} \tag{NUE} \label{NUE}
			\limsup_{n\rightarrow+\infty}\frac{1}{n}\sum_{i=0}^{n-1}\log\|Df(f^{j}(x))^{-1}\|\leq-c<0.
		\end{equation}
The non-uniform expansion on the set $H$ was explored through the following concept introduced in \cite{Alves} and generalized in \cite{ABV00}. We say that $n$ is a \textit{hyperbolic time} for $x$ if 
$$\prod_{j=n-k}^{n-1}\|Df(f^{j}(x))^{-1}\|\leq e^{-ck/2},\quad\text{for all }1\leq k< n.$$

The property~(\ref{NUE}) above is enough to guarantee positive frequency of hyperbolic times for points in $H$. The next result shows that the iterates of a map at hyperbolic times behave locally as uniformly expanding maps; we refer the reader to \cite[Lemma~5.2 ]{ABV00} for the proof.
\begin{lemma}\label{distortion} There exist $\delta>0$ such that for every  $0<\varepsilon\le\delta$ and every hyperbolic time~$n$ for $x\in M$ the dynamic ball $B_{\varepsilon}(x,n)$ is mapped diffeomorphically under $f^{n}$ onto the ball $B(f^{n}(x),\varepsilon)$. Moreover, for $y, z\in B_{\varepsilon}(x,n)$ we have
$$
d(f^{n-j}(y), f^{n-j}(z))\leq e^{-cj/4}d(f^{n}(y), f^{n}(z))
$$ for each $1\le j\le n$;
\end{lemma}
Therefore, $f$ satisfies our conditions~(\ref{estrela}) and (\ref{H1}). Note that if $n$ is a hyperbolic time of $x\in H$ then for any H\"older continuous function $\varphi:M\to \mathbb{R}$ and all $y, z\in B_{\varepsilon}(x,n)$ we have
$$\sum_{i=0}^{n}|\varphi(f^i(y))-\varphi(f^i(z))|\leq V
 $$
 which shows that $\varphi$ satisfies our condition~(\ref{H2}).      
		Moreover, in~\cite{OV08, VV10} it was proved that for a H\"older continuous potential $\varphi$ with \emph{small variation}, i.e. $$\sup\varphi - \inf\varphi < \log\deg(f)-\log q,$$  the relative pressure $P(\varphi, H)$ satisfies 
		$$P_{f}(\varphi, H^{c})<P_{f}(\varphi, H)=P_{f}(\varphi).$$
	Thus $\varphi$ also satisfies our condition~(\ref{H3}).} 
        \end{example}

In addition to the class presented above, we can apply our results in the following example, where the contraction rate is more general than in~\cite{OV08, VV10}.

\begin{example}\label{deformacao}\normalfont{Let $g:M\rightarrow M$ be an expanding map and $\mathcal{P}=\left\{P_{1},\dots,P_{n}\right\}$ be a covering of $M$ by injectivity domains of $g$. On each $P_{i}$ consider a fixed point $p_{i}$ and for $0<\delta< 1$ deform $g$ on a neighborhood $B_i=B(p_ i, \delta)$ of $p_{i}$ by a Hopf bifurcation or Pitchfork bifurcation to obtain a local homeomorphism $f:M\rightarrow M$ which coincides with the expanding map $g$ outside $B_{1}\cup\cdots\cup B_{n}$. Denoting by $\lambda\in (0, 1)$ the contraction rate of $f$ on $B_i$ we have for each $n\in\mathbb{N}$, $0<\vep<\delta$ and $y\in B_{1}\cup\cdots\cup B_{n}$ that %$\vep>d(y,f^n(p))>\vep \lambda$ and $x=f^{-n}(y)$ there exist  $0\le  j < n$
$$
d(f^n(x),f^n(y))\leq \lambda^{n} d(x,y)
$$
and so $B(f^n(y),\vep)\nsubseteq f^n(B_\vep(y,n))$. Moreover, if the deformation satisfies 
\begin{flushleft}
\begin{enumerate}
\item $f(B_{1}\cup\cdots\cup B_{n})\subset B_{1}\cup\cdots\cup B_{n};$
\item $h_{top}(f, B_{1}\cup\cdots\cup B_{n} )< h_{top}(f, (B_{1}\cup\cdots\cup B_{n})^{c})$
\end{enumerate}
\end{flushleft}
then the $f$-invariant set $$\cG=\bigcap_{j\geq 0}{f^{-j}((B_{1}\cup\cdots\cup B_{n} )^{c})}$$ is non-uniformly expanding and thus $f$ satisfies~ (\ref{estrela}) and (\ref{H1}). Since by construction it holds $$h_{top}(f, \cG^{c})<h_{top}(f,\cG)$$
%Note that, in this case, the accumulation points of the orbit of every point $x\in\cG^c$ is contained on a closed set of $\cG{c}.$
then any continuous potential $\varphi: M\rightarrow\mathbb{R}$ with the Bowen property~(\ref{H2}) on the set $\cG$ such that $$\sup\left\{\varphi(x);\, x\in \cG^{c}\right\}\leq\inf\left\{\varphi(x);\, x\in \cG\right\}$$ satisfies our condition~(\ref{H3}). In fact,
$$P_{f}(\varphi, \cG^{c})\leq h_{top}(f, \cG^{c})+\sup_{x\in \cG^{c}}{\varphi(x)} <h_{top}(f, \cG)+\inf_{x\in \cG}{\varphi(x)}\leq P_{f}(\varphi, \cG).$$
Therefore we can apply our results in this setting.

\begin{figure}[htb]
\begin{minipage}[b]{0.87\linewidth}
\begin{center}
\includegraphics[width=3.9cm]{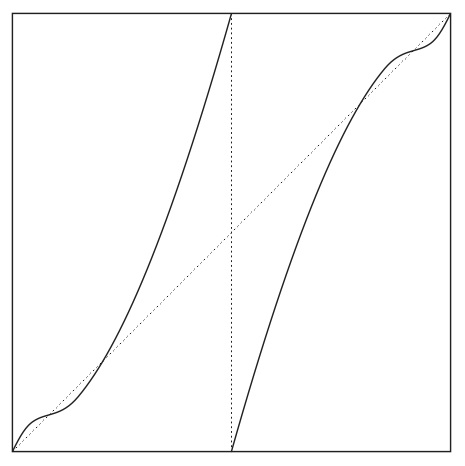}
\hspace{0.8cm}
\includegraphics[width=3.9cm]{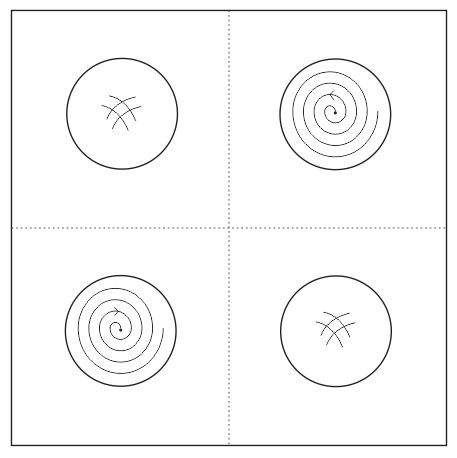}
\vspace{-0.4cm}
\caption{\!Deformations satisfying Example~\ref{deformacao}.}
\end{center}
\end{minipage}
\end{figure}}
\end{example}

In the following example, we present a broad class of dynamical systems introduced by Pinheiro in~\cite{Pin11}, where the concept of {\textit {zooming times}} is defined. This notion generalizes and relaxes the geometric features of hyperbolic times, allowing more flexibility for applications and examples. More recently, Pinheiro and Varandas~\cite{PinVar} studied the thermodynamic formalism for these systems, with a focus on the set of expanding measures.

\begin{example}[Zooming systems]\normalfont{
Consider a sequence  $\alpha=\{\alpha_n\}_{n\in\mathbb{N}}$ of functions $\alpha_n : [0, +\infty) \rightarrow  [0, +\infty)$ satisfying 
    \begin{itemize}
        \item $\alpha_n(r)\le r$ for all $r>0$ and $n\in \N$;
        \item $\alpha_n(r)<\alpha_n(\tilde{r})$ for all $0\leq r\le\tilde{r}$ and $n\in \N$;
        \item $\alpha_n\circ\alpha_m(r)\leq \alpha_{n+m}(r)$ for all $r>0$ and $n, m\in \N$;
        \item $\sup_{0\leq r\leq 1} (\sum_{n=1}^{\infty} \alpha_n(r))<\infty$.
    \end{itemize} 
    Let $f:X\to X$ be a measurable map defined on a connected, compact, separable metric space $(X, d)$ and let $\delta$ be a positive constant.
    \begin{definition} We say that $n\geq 1$ is a $(\alpha, \delta)$-zooming time for $p\in X$ (with respect to $f$) if there exists a neighborhood $V_n(p)$ of $p$ satisfying
    \begin{enumerate}
    \item $f^n$ sends $V_n(p)$ homeomorphically onto the ball $B_n(f^n(p), \delta).$
\item $d(f^{j}(z),f^{j}(y))\le \alpha_{n-j} \,(d(f^{n}(z),f^{n}(y)))$
    for every $y,z \in V_n(p)$ and every $ 1\le j\le n$.

    \end{enumerate}
    \end{definition}
    A positively invariant $\Lambda\subset X$ is called zooming set if every $x\in \Lambda$ has positive frequency of zooming times.
    Therefore, in this context, we observe that any local homeomorphism $f:X\to X$ with the zooming set $\Lambda$ satisfies our conditions~(\ref{estrela}) and~(\ref{H1}).}
\end{example}

As a subclass of the previous example, we consider the system described in~\cite{Pin11}, which provides a zooming but not expanding map.  
\begin{example}\label{exa nue top}\normalfont{Let $\sigma:\Sigma_2^+\rightarrow \Sigma_2^+$ be the one-side shift with the metric 
    \begin{align*}
        d(x,y) =
\left \{
\begin{array}{cc}
0, & \mbox{if} \;\; x=y\\
n(x,y)^{-2}, &\, \mbox{if}\;\; x\neq y .\\
\end{array}\right.
    \end{align*}
    where $x=(x_n)_n$, $y=(y_n)_n$ and $n(x,y):=\min\{n\ge 1: x_n\neq y_n\}$. Notice that $\sigma$ is a conformal map in the following sense:    $$D\sigma(x):=\lim_{y\to x}\frac{d(\sigma(y),  \sigma(x))}{d(y,x)}=1$$
     for all $x\in \Sigma_2^+$. In particular, we have 
    $$
    \lim_{n\rightarrow \infty}\frac{1}{n}\log D\sigma^n(x)=\lim_{n\rightarrow \infty}\frac{1}{n}\sum_{j=0}^{\infty}\log D\sigma(\sigma^j(x))=0.
    $$
    for all $x\in \Sigma_2^+$.  Therefore, $$\nu\left(\{x \in \Sigma_2^+\,;\,\,  \limsup_{n\rightarrow \infty}\frac{1}{n}\sum_{n=0}^{\infty}\log D\sigma(\sigma^n(x))>0\}\right)=0,$$ for any measure $\nu\in \mathbb{P}_{\sigma}(\Sigma_2^+)$, which shows that there are no expanding measures or expanding sets. Consequently, no potential $\varphi: \Sigma_2^+ \rightarrow \mathbb{R}$ is hyperbolic in the sense of \cite{RV17}.
    
   Let $C(x,n)=\{w\in \Sigma_2^+: w_i=x_i\,,\, 1\le i\le n \}$ be the $n$-cylinder of $x\in \Sigma_2^+ $. Since $\sigma^n(C(x,n))=C(\sigma^n(x),1)$ for all $x\in \Sigma_2^+$ and $n\in \N$ it follows that $\cG=\Sigma_2^+$ and thus, $\sigma$ satisfies our condition~(\ref{estrela}) and~(\ref{H1}). 
   
   Moreover, as we have  $$n(\sigma^j(y),\sigma^j(z))=n(\sigma^n(y),\sigma^n(z))+(n-j),$$
   for every $y, z\in C(x,n)$ and $0\leq j\leq n$, we obtain that 
   $$    d(\sigma^j(y),\sigma^j(z))=\alpha_{n-j}( d(\sigma^n(x),\sigma^n(y))),
    $$
    where $\alpha_n(r)=(1+n\sqrt{r})^{-2}r$. 
 Therefore, any H\"older continuous potential $\varphi: \Sigma_2^+ \rightarrow \R$ is Bowen and satisfies the condition~(\ref{H3}).}
    
\end{example}

Finally, as an application of our results for attractors sets, we present a general class of skew-products that satisfies our hypothesis. This class generalizes the setting of \cite{CN}, \cite{FO} and \cite{RV17}; in particular, it includes attractors derived from solenoid-like systems and the family of partially hyperbolic horseshoes introduced by~\cite{diazetal} and studied by~\cite{RS17, RS16} . 

\begin{example} \normalfont{Consider $f:M\rightarrow M$ be a local homeomorphism as in the Examples~\ref{paulo}, \ref{deformacao} or \ref{exa nue top}. Let $K$ be a compact metric space with distance $d_K$ and $g:M\times K\rightarrow K$ be an endomorphism that is a uniform contraction on $K$: there exists  $\lambda\in(0, 1)$ such that for all $x\in M$  and  $y_1,y_2\in K$ we have
\begin{equation*}
d_K\big(g(x, y_1), g(x,y_2)\big)\leq\lambda\, d_{K}(y_1,y_2).
\end{equation*} 
Define on $M\times K$ the skew-product:
$$F:M\times K\rightarrow M\times K,\;\;\;F(x,y)=(f(x),\,g_{x}(y)\,).$$
Notice that the continuous conjugacy between $F$ and $f$ is the canonical projection $\pi$ on the first coordinate. 

Let $\phi: M\times K\rightarrow\mathbb{R}$ be a H\"older continuous potential. If the base dynamic $f$ is as described in Example~\ref{paulo}, we can consider $\phi$ with small variation: $$\sup\phi - \inf\phi < \log\deg(f)-\log q.$$ 
and for $f$ as in Example~\ref{deformacao}, we can consider $\phi$ satisfying $$\sup\left\{\phi(x);\, x\in \hat{\cG}^{c}\right\}\leq\inf\left\{\phi(x);\, x\in \hat{\cG}\right\}.$$ 
In any case, $\phi$ satisfies our conditions~(\ref{H4}) and (\ref{H5}). Indeed, in Section~\ref{atrator}, we prove that any H\"older continuous potential $\phi: M\times K\to \mathbb{R}$ is cohomologous to one which does not depend on the stable direction; Moreover, we can induce a continuous potential $\varphi: M\to\mathbb{R}$ for $f$ whose variation is related to the variation of $\phi$. Therefore, as in Examples~\ref{paulo}, \ref{deformacao} or \ref{exa nue top}, it follows that $\varphi$ is a Bowen potential whose pressure is located on $\cG.$ These properties imply that $\phi$ satisfies~(\ref{H4}) and (\ref{H5}).} 
\end{example}

\section{Topological Pressure} \label{toppressure} 
In this section, we present the definition and some properties of topological pressure.
Let $T:X\to X$ be a continuous transformation defined on a compact metric space $X$ and $\phi:X\rightarrow\mathbb{R}$ a continuous function. Given an open cover $\alpha$ of $X$, the pressure $P_{T}(\phi, \alpha)$ of $\phi$ with respect to $\alpha$ is 
$$P_{T}(\phi, \alpha):=\lim_{n\rightarrow +\infty}{\frac{1}{n}\log\inf_{\mathcal{U}\subset\alpha^{n}}\{\sum_{U\in\mathcal{U}}e^{\phi_{n}(U)}\}}\\$$
where the infimum is taken over all subcover $\mathcal{U}$ of $\alpha^{n}=\vee_{n\geq 0}T^{n}\alpha$ and $\phi_{n}(U)$ is the supremum $\sup_{x\in U}\sum_{j=0}^{n-1}\phi\circ T^{j}(x).$ 

\begin{definition} The topological pressure $P_{T}(\phi)$ of $\phi$ with respect to $T$ is 
$$P_{T}(\phi):=\lim_{\delta\rightarrow 0}\big\{\sup_{|\alpha|\leq\delta} P_{T}(\phi,\alpha)\big\}$$ 
where $|\alpha|$ denotes the diameter of the open cover $\alpha.$
\end{definition}

The connection between topological pressure and the metric entropy of an invariant measure is formalized by the variational principle. Let $\mathbb{P}_T(X)$ be the set of $T$-invariant probability measures. The variational principle asserts that the topological pressure of $(T, \phi)$ is given by
\begin{equation}\label{pressure2}
P_T(\phi) = \sup_{\eta \in \mathbb{P}_T(X)} \left\{ h_{\eta}(T) + \int \phi \, d\eta \right\}.
\end{equation}
We refer the reader to \cite[Theorem 4.1]{W} for a proof. In particular, the variational principle provides a natural way for selecting meaningful measures in $\mathbb{P}_T(X)$. A measure $\mu_\phi \in \mathbb{P}_T(X)$ is called an \textit{equilibrium state} for $(T, \phi)$ if it attains the supremum in~(\ref{pressure2}), thereby realizing the topological pressure:
$$P_{T}(\phi)=h_{\mu_{\phi}}(T)+\int{\phi}\, d\mu_{\phi}.$$

An alternative way to define topological pressure is through the notion of relative pressure. This approach is from dimension theory and it is very useful to calculate the topological pressure of non-compact sets. The reader can see  more property of the relative topological pressure  in~\cite{Pes}.

Given $\vep>0$, $n\in\mathbb{N}$ and $x\in X$, consider the \emph{dynamical ball}  
$$B_{\vep}(x,n)=\left\{y\in X:\;d(T^{j}(x), T^{j}(y))<\vep,\;\mbox{for}\;\; 0\leq j\leq n\right\}.$$
For each  $N\in\mathbb N$ consider the family
 $$\mathcal{F}_{N}=\{B_{\vep}(x,n):\; x\in X\;\mbox{and}\; n\geq N\}.$$
Let $\Lambda\subset X$ be a positively invariant set ($T(\Lambda)\subset \Lambda$), not necessarily compact. We denote by $\mathcal{F}_{N}(\Lambda)$ the finite or countable families of elements in $\mathcal{F}_{N}$ which cover $\Lambda.$
Define for $n\in\mathbb N$
 $$S_n\phi(x)=\phi(x)+\phi(T(x))+\cdots+\phi(T^{n-1}(x))
 $$
 and
 $$R_{n,\vep}\phi(x)=S_{n}\phi(B_{\vep}(x, n))=\sup_{y\in B_\vep(x,n)}S_n\phi(y).$$
For each~$\gamma>0$ we consider
$$ m_{T}(\phi, \Lambda, \vep, N,\gamma)=
\inf_{\mathcal{U}\subset\mathcal{F}_{N}(\Lambda)} \left\{\sum_{B_{\vep}(x,n)\in\mathcal{U}}e^{-\gamma n+ R_{n,\vep}\phi(x)}\right\}.
$$
Define
$$m_{T}(\phi, \Lambda, \vep, \gamma)=\lim_{N\rightarrow +\infty}{m_{T}(\phi, \Lambda, \delta, N,\gamma)},$$
and by taking the infimum over $\gamma$ we obtain
$$P_{T}(\phi, \Lambda, \vep)=\inf{\{\gamma >0 \, | \; m_{T}(\phi, \Lambda, \vep,\gamma)=0\}}.$$
Therefore the \textit{relative topological pressure} of $\phi$ %$P_{T}(\phi, \Lambda)$ of a subset 
on a set $\Lambda$ is defined by
 $$P_{T}(\phi, \Lambda)=\lim_{\vep\rightarrow 0} {P_{T}(\phi, \Lambda, \vep)}.$$
In particular,  the \emph{topological pressure of $\phi$} is $P_T(\phi, X)$, and it satisfies
\begin{equation}\label{pressure}
P_{T}(\phi)=\sup\left\{ P_{T}(\phi, \Lambda),\, P_{T}(\phi, \Lambda^{c})\right\},
\end{equation}
where $\Lambda^c$ is the complement of the set $\Lambda$ on $X$. 

Let $C^0(X)$ be the space of real continuous functions $\phi:X\to \mathbb {R}$ endowed with the supremum norm. Fix a positively invariant set $\Lambda \subset X$ and define the set $\mathcal{P}(\Lambda)$ as follows $$\mathcal{P}(\Lambda):=\{\phi \in C^{0}(X);\,\, P_{T}(\phi, \Lambda^c)< P_{T}(\phi, \Lambda)\}.$$
In the next proposition, we show that $\mathcal{P}(\Lambda)$ is an open subset of the space of continuous functions in the $C^0$-topology.
In particular, the set $\mathcal{P}(\cG)$ of functions satisfying condition~(\ref{H3}) is also open in the space of continuous functions endowed with the $C^0$-topology. 

\begin{proposition}\label{pr.openess} 
Given $\phi\in \mathcal{P}(\Lambda)$ there exists $\zeta>0$ such that for each $\psi\in C^{0}(X)$ with $\|\psi-\phi\|<\zeta$ we have $P_T(\psi, \Lambda^c)< P_T(\psi,\Lambda)=P_T(\psi).$
\end{proposition}

\begin{proof} Given $\phi\in C^{0}(X)$ satisfying $P_T(\phi, \Lambda^c)<P_T(\phi, \Lambda)=P_T(\phi)$, we fix $\varepsilon_\phi>0$ such that $P_T(\phi, \Lambda^c)<P_T(\phi)-\varepsilon_\phi.$
Since the topological pressure depends continuously on the potential $\phi\in C^0(X)$ (see e.g. \cite{Wal00}), then for $0<\varepsilon<\varepsilon_\phi$ we can find $\zeta>0$ such that for any $\psi\in C^{0}(X)$ with $\|\psi-\phi\|<\zeta$ we have $P_T(\psi)\geq P_T(\phi)-\varepsilon.$
%Let $\mathcal{U}$ be a finite or
%countable family of elements in $\mathcal{F}_N$ which cover the set $\Lambda^c$. 
For every $\gamma\in\mathbb{R}$ we have
\begin{eqnarray*}
m_{T}(\psi, \Lambda^c, \delta, N,\gamma)&=&\inf_{\mathcal{U}\subset\mathcal{F}_{N}(\Lambda^c)} \left\{\sum_{B_{\delta}(x,n)\in\mathcal{U}}e^{-\gamma n+ R_{n,\delta}\psi(x)}\right\}\\
&\le&\inf_{\mathcal{U}\subset\mathcal{F}_{N}(\Lambda^c)} \left\{\sum_{B_{\delta}(x,n)\in\mathcal{U}}e^{-\gamma n+\zeta n+ R_{n,\delta}\phi(x)}\right\}\\
&=&m_{T}\left(\phi, \Lambda^c, \delta, N,\gamma -\zeta\right),
\end{eqnarray*}
and therefore 
$$P_T(\psi, \Lambda^c)\leq P_T(\phi, \Lambda^c)-\zeta<P_T(\phi)-\varepsilon_{\phi}- \zeta\leq P_T(\phi)-\varepsilon\leq P_T(\psi).$$
Hence, for any  $\psi\in C^{0}(X)$ with $\|\psi-\phi\|<\zeta$ we have $$P_T(\psi, \Lambda^c)< P_T(\psi,\Lambda)=P_T(\psi),$$
which gives the desired conclusion.   %Let us now see the general case.
%Given $\psi\in C^0(X)$ consider $\tilde\psi=\psi-\inf\psi$. It easily follows that
% $$P_T(\tilde \psi)=P_T(\psi)-\inf\psi$$
%Take $\phi\in C^{0}(X)$ satisfying $P_T(\phi, \Lambda^c)<P_T(\phi, \Lambda)=P_T(\phi)$. As $\tilde{\phi}$ is non-negative, using the part already seen we easily get
%\begin{eqnarray}\label{desigualdade do potencial}
%P_T(\tilde{\phi}, \Lambda^c)\leq P_T(\phi, \Lambda^c)-\inf\phi<P_T(\phi)-\inf\phi=P_T(\tilde{\phi}).
%\end{eqnarray}
%From the previous computations it follows that there exists a neighborhood of $\tilde{\phi}$ where for any $\tilde{\psi}$ holds the inequality~(\ref{desigualdade do potencial}). Therefore the same is true in a neighborhood of $\phi.$ 
\end{proof}

We finish this section by proving that cohomologous potentials have the same relative topological pressure. As a consequence, they have the same topological pressure and the same equilibrium states. We recall that two potentials $\phi\, ,\psi: X\to \mathbb R$ are  \emph{cohomologous} if there exists a continuous function $u:X \to\mathbb{R}$ such that $\psi=\phi-u+u\circ T$; 

\begin{proposition}\label{pressao cohomologo}
Let $T:X\rightarrow X$ be a continuous transformation on a compact metric space. If $\phi\, ,\psi: X\to \mathbb R$ are cohomologous potentials then for any positively invariant set $\Lambda\subset X$ we have that  $P_{T}(\Lambda,\phi)=P_{T}(\Lambda,\psi)$.
\end{proposition} 
\begin{proof}
Let $\mathcal{F}_{N}(\Lambda)$ be the finite or countable families of elements in $\mathcal{F}_{N}$ that cover $\Lambda$, where $\mathcal{F}_{N}$ is the collection of dynamical balls 
 $$\mathcal{F}_{N}=\{B_{\vep}(x,n):\; x\in N\;\mbox{and}\; n\geq N\}.$$
Since  $\psi=\phi-u+u\circ T$ we have $S_n\psi(x)=S_n\phi(x)-u(x)+u\circ T^n(x)$ for any $n\in \mathbb{N}$. Let $\|u\|=\sup_{x\in X} |u(x)|$. Then
$$|R_{n,\vep}\psi(x)-R_{n,\vep}\phi(x)|=|\sup_{y\in B_\vep(x,n)}S_n\psi(y)-\sup_{y\in B_\vep(x,n)}S_n\phi(y)|\leq 2\|u\|$$
and thus for any $\mathcal{U}\subset\mathcal{F}_{N}(\Lambda)$ we have
\begin{equation}\label{inf}
e^{-2\|u\|}\sum_{\mathcal{U}}e^{-\gamma n+ R_{n,\vep}\phi(x)}\leq \sum_{\mathcal{U}}e^{-\gamma n+ R_{n,\vep}\psi(x)}\leq e^{2\|u\|}\sum_{\mathcal{U}}e^{-\gamma n+ R_{n,\vep}\phi(x)}.
\end{equation}
Recalling that by definition
$$
m_{T}(\phi, \Lambda, \vep, N,\gamma)=
\inf_{\mathcal{U}\subset\mathcal{F}_{N}(\Lambda)} \left\{\sum_{B_{\vep}(x,n)\in\mathcal{U}}e^{-\gamma n+ R_{n,\vep}\phi(x)}\right\},
$$
from inequality~(\ref{inf}) we obtain
\begin{equation*}
    e^{-2\|u\|}m_{T}(\phi, \Lambda, \vep, N, \gamma)\leq m_{T}(\psi,  \Lambda, \vep, N, \gamma)\leq e^{2\|u\|} m_{T}(\phi,  \Lambda, \vep, N, \gamma).
\end{equation*}
As $N$ goes to infinity it follows that 
\begin{equation}\label{m}
    e^{-2\|u\|}m_{T}(\phi,  \Lambda, \vep, \gamma)\leq m_{T}(\psi, \Lambda, \vep, \gamma)\leq e^{2\|u\|} m_{T}(\phi, \Lambda, \vep, \gamma)
\end{equation}
for all $\gamma>0$. Recalling that
$$P_{T}(\phi, \Lambda, \vep)=\inf{\{\gamma >0 \, | \; m_{T}(\phi,  \Lambda, \vep,\gamma)=0\}},$$
we  take the infimum over $\gamma$ in~(\ref{m}) and conclude for all $\vep>0$ that
$$P_T(\phi, \Lambda, \vep)=P_T(\psi,  \Lambda, \vep).$$ 
Thus $P_T(\phi, \Lambda)=P_T(\psi, \Lambda),$ for any positively invariant set $\Lambda\subset X.$ 
%In particular, if $\phi$ is hyperbolic we have
%$$P_F(\bar\phi, (\hat{\Sigma}(F))^c)=P_F(\phi, (\hat{\Sigma}(F)^c)<P_F(\phi, \hat{\Sigma}(F))=P_F(\bar\phi, \hat{\Sigma}(F))=P_F(\bar\phi)$$
%and this implies that $\bar\phi$ is hyperbolic.
\end{proof}

%%%%%%%%%%%%%%%%%%%%%%%%%%%%%%%%%%%%%%%%%%%

 \section{The Reference measure}\label{reference measure}  
 %Let $C^{0}(M)$ be the space of real continuous functions $\psi:M\to \mathbb {R}$ endowed with the sup norm. 
 The \emph{Ruelle-Perron-Fröbenius transfer operator} associated to the local homeomorphism  $f:M\rightarrow M$ and the real continuous function $\varphi:M\rightarrow \mathbb{R}$ is the linear operator
 $$\mathcal{L}_{f,\varphi} : C^{0}(M) \rightarrow C^{0}(M)$$ which associates to $\psi \in C^{0}(M)$ the continuous function
$ \mathcal{L}_{f,\varphi} (\psi) \colon M\to \mathbb{R}$ defined by
$$\label{optransf}
\mathcal{L}_{f,\varphi} \psi \left(x\right) = \displaystyle\sum _{y  \in \, f^{-1}\left(x\right)} e^{\varphi(y)} \psi(y). 
$$
Notice that $ \mathcal{L}_{f,\varphi} $ is a bounded linear operator and for each $n \! \in\! \mathbb{N}$ we have  
$$\label{iteradostransf}
\mathcal{L}_{f,\varphi}^{n} \psi (x) = \displaystyle\sum _{y  \in \, f^{-n}\left(x\right)} e^{S_{n}\varphi(y)} \psi \left( y \right),
$$
where $S_{n}\varphi$ denotes the Birkhoff sum $S_{n}\varphi(x)= \displaystyle\sum_{j=0}^{n-1} \varphi\big(f^{j}(x)\big).$

Since we are considering $\mathcal{L}_{f,\varphi} : C^{0}(M) \rightarrow C^{0}(M)$, by Riesz-Markov Theorem, we may think of  its dual operator $\mathcal{L}_{f,\varphi}^{\ast}: \mathbb{P}(M) \to \mathbb{P}(M)$. Moreover,  we have
$$\int \psi\ d\mathcal{L}_{f,\varphi}^{\ast}\eta  = \int  \mathcal{L}_{f,\varphi}(  \psi ) \ d\eta , $$
for every $\psi \in  C^{0}(M) $ and every $\eta \in \mathbb{P}(M)$. It is a classical result that
for each $(f, \varphi)$  there exists a probability measure $\nu$ satisfying $$\mathcal{L}_{f,\varphi}^{\ast}\nu= \lambda \nu ,$$
where $\lambda$ is the spectral radius of $\mathcal{L}_{f,\varphi} $. The reader can see \cite[Lema 3.2]{ARS18} for a proof of this result. We refer to $\nu$ as the \emph{reference measure} associated to $(f, \varphi)$. Next we will show some important properties of this measure.

The \textit{Jacobian} of a probability measure $\mu$ with respect to $f$ is a measurable function $J_{\mu}f$ satisfying
$$\mu(f(A))=\int_{A}J_{\mu}f d\mu,$$
for any measurable set $A$ where $f|_{A}$ is injective. For the reference measure, it is straightforward to check that the jacobian is $J_{\nu}f=\lambda e^{-\varphi}.$ Hence
$$\nu(f(A))=\int_{A}{\lambda e^{-\varphi}}d\nu,$$
for any measurable set $A$ where $f|_{A}$ is injective. Notice that $J_{\nu}f$ is a continuous function bounded away from zero and infinity. This good Jacobian allows to prove the \emph{Gibbs property} of $\nu$ at times satisfying~(\ref{estrela}). 

\begin{lemma}\label{conforme}
% Let $f:M\to M$ be a local homeomorphism satisfying~(\ref{H1}) and let ${\varphi}: M \to \mathbb{R}$ be a continuous potential satisfying~(\ref{H2}) and (\ref{H3}). 
Let $\nu$ be the reference measure of $(f, \phi).$ Then $\supp(\nu)=M$. Moreover, for every $0<\vep< \delta_0/2$, there exists $K=K(\vep)>0$ such that if $x\in\cG$ and $n=n(x)\in\mathbb{N}$ satisfies~(\ref{estrela}) then $$K^{-1} \leq  \frac{\nu(B_\varepsilon(x, n))}{\exp({S_n\varphi(y) - n P})}  \leq K$$ 
for all $y\in B_{\varepsilon}(x,n)$, where $P=\log\lambda.$
\end{lemma}
\begin{proof} We begin by proving that $\nu$ is an open measure i.e. positive in open sets. In fact, the density of pre-orbit $\{f^{-n}(x)\}_{n\geq 0}$ of points $x\in M$ implies the inclusion $M\subset\bigcup_{s\in\mathbb{N}}f^{s}(U)$ for any open set $U\subset M$.
Since for each $s$, $f^{s}$ is also a local homeomorphism we can decompose $U$ into subsets $V_{i}(s)\subset U$ such that $f^{s}|_{V_{i}(s)}$ is injective. Hence,
\begin{eqnarray*}
1=\nu(M)\leq\sum_{s}\nu(f^{s}(U))
&\leq&\displaystyle\sum_{s}\sum_{i}\int_{V_{i}(s)}{\lambda^{s}e^{-S_{s}\varphi(x)}}d\nu\\ 
&\leq&\displaystyle\sum_{s}\lambda^{s}\displaystyle\sum_{i}\sup_{x\in V_{i}(s)}(e^{S_{s}\varphi(x)})\nu(V_{i}(s)).
\end{eqnarray*}
Thus there exists some $V_{i}(s)\subset U$ such that $\nu(U)\geq\nu(V_{i}(s))>0.$ This implies that $\supp(\nu)=M.$

Given $x\in \cG$, let $n(x)\in \mathbb N$ satisfying~(\ref{estrela}), i.e.,  $f^{n}(B_{\varepsilon}(x,n))=B(f^n(x),\vep)$ for any $0< \vep<\delta_0$. Since $\supp(\nu)=M$ we can find a uniform constant $\xi_\varepsilon>0$ depending on the radius $\varepsilon$ such that 
\begin{eqnarray*}
\xi_\varepsilon\leq\nu(B(f^{n}(x),\varepsilon))=\nu(f^{n}(B_{\varepsilon}(x,n)))=\int_{B_{\varepsilon}(x,n)}{\hspace{-0.4cm}\lambda^{n}e^{-S_{n}\varphi(y)}}\,d\nu\leq1.
\end{eqnarray*} 
Since $\varphi$ satisfies the Bowen property on the set $\cG$ we obtain
\begin{eqnarray*}
\xi_\varepsilon\leq\int_{B_{\varepsilon}(x,n)}\hspace{-0.4cm}\lambda^{n}e^{-S_{n}\varphi(y)} \frac{e^{-S_{n}\varphi(z)}}{e^{-S_{n}\varphi(y)}}\,d\nu
\!\!\!&\leq&\!\!\! V e^{-S_{n}\varphi(y)+n\log\lambda}\nu(B_{\varepsilon}(x,n))
\end{eqnarray*}
and
\begin{align*}
    e^{-S_{n}\varphi(y)+n\log\lambda}\nu(B_{\varepsilon}(x,n))\le V\int_{B_{\varepsilon}(x,n)}\hspace{-0.4cm}\lambda^{n}e^{-S_{n}\varphi(y)} \frac{e^{-S_{n}\varphi(z)}}{e^{-S_{n}\varphi(y)}}\,d\nu.
\end{align*}
Then, for $K=\max\{V, \xi_\vep^{-1}V\}$ we obtain the Gibbs property
$$K^{-1} \leq  \frac{\nu(B_\varepsilon(x, n))}{\exp({S_n\varphi(y) - n P})}  \leq K$$ 
for all $y\in B_{\varepsilon}(x,n)$ and $n$ satisfying~(\ref{estrela}), where $P=\log\lambda.$
\end{proof}

The last lemma showed that every iterate satisfying condition~(\ref{estrela}) is a Gibbs time for $x\in \cG$. Therefore, from now on, we will call the sequence $(n_{k}(x))_k$ given by condition~(\ref{estrela}) as Gibbs times.

\begin{corollary}\label{gibbs_n}
 Given $x\in \cG$, suppose that the sequence of Gibbs times $(n_i(x))_i$ is non-lacunar. Then, for each $0<\vep<\delta_0/2$ and $n\geq n_1(x)$, if $B_{\vep}(y,n)\subset  B_{\delta_0/2}(x,n)$, we have 
$$
K^{-1}e^{-\alpha(n_{i+1}-n_i(x))}\le \frac{\nu(B_\varepsilon(y, n))}{\exp\left(S_n\varphi(y) - nP\right)} \le Ke^{\alpha(n_{i+1}-n_i(x))}
$$
where $\alpha=\sup|\varphi|+P$ and $n_i(x)$, $n_{i+1}(x)$ are consecutive Gibbs times such that $n_i(x)\le n \le n_{i+1}(x)$. Moreover, $\displaystyle\lim_{n\to+\infty} \frac{1}{n}\log Ke^{\alpha(n_{i+1}-n_i(x))}=0$.
\end{corollary}
\begin{proof} For any $n\in\mathbb{N}$ we can find $n_i(x)$ such that $n_i(x)\le n \le n_{i+1}(x)$. Applying Lemma~\ref{conforme}, if $B_{\vep}(y,n)\subset  B_{\delta_0/2}(x,n)$ we have that 
    $$
    \nu(B_\varepsilon(y, n))\le K \exp\left(\sum_{j=0}^{n-1}\varphi(f^j(y)) - nP\right)e^{(\sup |\varphi|+P)(n_{i+1}-n_i(x))}.
    $$
The lower bound follows in a similar way. Taking $\alpha=\sup |\varphi|+P$ and using the fact that $(n_i(x))_i$ is non-lacunar, we complete the proof.
   \end{proof}
%Note that
 %  $$
  % \frac{1}{n}\log  K\exp(\alpha(n-n_i(x))))\le \frac{n_{i+1}-n_i}{\gamma(n_i)}+ %O(\frac{1}{n}).
   %$$
%As the sequence $n_i$ is $\gamma$-lacunar, the left side converges to zero. 

The definition of weak Gibbs measures depends on the constant $P$.  In the next proposition, we prove that $P$ is uniquely determined by the \emph{relative pressure} $P_{f}(\varphi, \cG)$ of $\varphi$ on $\cG$. %This is a fundamental step to answer the question \textit{"when a weak Gibbs measure is an equilibrium state?"}.
\begin{proposition} \label{zooming measure}
Let $f:M\to M$ be a local homeomorphism satisfying~(\ref{H1}) and let ${\varphi}: M \to \mathbb{R}$ be a continuous potential satisfying~(\ref{H2}) and (\ref{H3}). Then $P=P_f(\varphi)$ and $\nu(\cG)=1$.
\end{proposition} 
\begin{proof}
Recall that $P=\log\lambda$ where $\lambda$ is the spectral radius of $\mathcal{L}_{f,\varphi}.$ We will show that $\log\lambda=P_f(\varphi).$ The inequality $\log\lambda\leq P_{f}(\varphi)$ is a consequence of $\nu$ to be a reference measure i.e. $\mathcal{L}_{f,\varphi}^{\ast}\nu= \lambda \nu$. Given a cover $\alpha$ of $M$ with diameter less than $\delta,$  let $\mathcal{U}$ be a subcover of $\alpha^{n},$ for $n\in \mathbb{N}$. We have that
\begin{eqnarray*} 
\lambda^{n}=\lambda^{n}\nu(M)=\int_{M}{\mathcal{L}^{n}_{\varphi}(\textbf{1})}d\nu\leq\sum_{U\subset\mathcal{U}}\int_{U}{e^{S_{n}\varphi(x)}}d\nu
%\leq\sum_{U\subset\mathcal{U}}e^{S_{n}\varphi(U)}\nu(U)
\leq\sum_{U\subset\mathcal{U}}e^{S_{n}\varphi(U)}.
\end{eqnarray*}
Taking the infimum of all subcover $\mathcal{U}$ of $\alpha^{n}$ we obtain the inequality 
$$\log\lambda\leq\frac{1}{n}\log\inf_{\mathcal{U}\subset \alpha^{n}}\Big\{\sum_{U\subset\mathcal{U}}e^{S_{n}\varphi(U)}\Big\}\;\;{\stackrel{n\rightarrow+\infty}{\longrightarrow}}\;\; P_{f}(\varphi,\alpha)\leq P_{f}(\varphi).$$ 

For the converse, recall that every point $x\in \cG$ has infinitely many Gibbs times. Then for $\varepsilon>0$ small we can fix $N>1$ sufficiently large such that
$$\cG\subset\bigcup_{n\geq N}\bigcup_{x\in \cG_n}B_{\varepsilon}(x,n),$$
where $\cG_n$ is the set of points that have $n$ as a Gibbs time. 
Since $f^n(B_\varepsilon(x,n))$ is the ball $B(f^n(x),\varepsilon)$ in $M$ thus, by Besicovitch Covering Lemma, there exists a countable family $F_{n}\subset \cG_{n}$ such that every point $x\in \cG_{n}$ is covered by at most $d=\dim(M)$ dynamical balls $B_{\varepsilon}(x,n)$ with $x\in F_{n}.$ Hence 
$$\mathcal{F}_{N}=\left\{B_{\varepsilon}(x,n) : x\in F_{n}\;\textbf{\rm{and}}\;n\geq N\right\},$$ 
is a countable covering of $\cG$ by dynamic balls with diameter less than $\varepsilon>0$.
Recalling the definition of relative pressure, take any $\gamma>\log\lambda$ and apply Lemma~\ref{conforme} to each element in $\mathcal{F}_{N}$ to deduce that 
 \begin{eqnarray*}
\sum_{n\geq N}\sum_{B_{\varepsilon}(x,n)\in \mathcal{F}_{N}}   e^{-\gamma n+R_{n,\varepsilon}\varphi(x) }\leq K\sum_{n\geq N}e^{-(\gamma-\log\lambda)n}
\leq  Ke^{-(\gamma-\log\lambda)N}.
\end{eqnarray*}
 Taking limit in $N$ we obtain 
$$m_{f}(\varphi, \cG, \varepsilon, \gamma)=\lim_{N\rightarrow +\infty}{m_{f}(\varphi, \cG, \varepsilon, N,\gamma)}=0.$$
Since $\varepsilon>0$ is arbitrary and $\gamma>\log\lambda$ we conclude that $P_{f}(\varphi,\cG)\leq\log\lambda.$
By condition~(\ref{H3}) we get $$\log\lambda\leq P_f(\varphi)=P_{f}(\varphi, \cG)\leq\log\lambda.$$
This finishes the first part of the proposition. For the second,  given $\varepsilon>0$ small, let  $\mathcal{B}_{N}$ be the collection of dynamical balls $$\mathcal{B}_{N}=\{B_{\varepsilon}(x,n)\,/\; x\in M\;\textbf{\rm{and}}\; n\geq N\}.$$
By the definition of jacobian of $\nu$, for every finite or countable family $\mathcal{U}$ of $\mathcal{B}_{N}$ which cover $\cG^{c}$, we have 
\begin{align*}
\nu(\cG^{c})\leq\nu(\mathcal{U})\leq\sum_{B_{\varepsilon}(x, n)\in\mathcal{U}}{\lambda^{-n}e^{S_{n}\varphi(B_{\varepsilon}(x, n))}}=\sum_{B_{\varepsilon}(x, n)\in\mathcal{U}}e^{-(\log\lambda) n+R_{n, \varepsilon}\varphi(x)}.
\end{align*}
Taking the infimum over $\mathcal{B}_{N}$, by the definition of relative pressure, it follows
$$\nu(\cG^{c})\leq m_{f}(\varphi, \cG^{c}, N, \log\lambda).$$
Since $P_{f}(\varphi, \cG^{c})< P_{f}(\varphi)=P_{f}(\varphi, \cG)=\log\lambda$, we obtain
$$\nu(\cG^{c})\leq \lim_{N\rightarrow +\infty}m_{f}(\varphi, \cG^{c}, N, \log\lambda)=m_{f}(\varphi, \cG^{c},\log\lambda)=0.$$
Thus $\nu(\cG)=1$.

\end{proof}

Notice that by the proof of Proposition~\ref{zooming measure}, $\lambda$ is the unique real eigenvalue of $\mathcal{L}_{f,\varphi}^{\ast}$ because $\lambda$ is uniquely determined by the topological pressure of the system $(f, \varphi)$. 

To conclude this subsection, we note that every $f$-invariant set contains a topological disk up to a set of zero $\nu$-measure. 
\begin{lemma} \label{disco} If $A$ is a $f$-invariant set with positive $\nu$-measure, then there exists a topological disk $\Delta$ of radius $\delta/4$ so that $\nu(\Delta\backslash A)=0$
\end{lemma}

For a proof of this result, see~\cite{VV10}.

\section{Weak Gibbs measure} \label{weak}
In the previous section, we proved the existence of a reference measure $\nu$ associated to the eigenvalue $\lambda=e^{P_{f}(\varphi)}.$ Moreover, we established the Gibbs property of $\nu$ for times satisfying the condition~(\ref{estrela}).
Here, we show the existence and uniqueness of an ergodic measure  $\mu_\varphi$ that is absolutely continuous with respect to $\nu$, and consequently, we conclude that $\mu_\varphi$ is the weak Gibbs measure stated in Theorem~\ref{Gibbs base}. Finally, we prove that $\mu_\varphi$ is the unique equilibrium state of $(f, \varphi)$, thereby establishing Theorem~\ref{ee base}.

\subsection{Absolutely continuous invariant measure} Recall that $\cG$ is the set satisfying the condition~(\ref{H1}). Let $\cG_{j}$ be the set of points whose have $j$ as Gibbs time. Consider the sequences $\left(\mu_{n}\right)_{n}$ and $(\eta_n)_{n}$  defined by  
$$\eta_{n}:=\frac{1}{n}\sum_{j=0}^{n-1}{f_{*}^{j}(\nu|_{\cG_{j}})}\;\; and \;\;\mu_{n}:=\frac{1}{n}\sum_{j=0}^{n-1}f_{*}^{j}(\nu|_{\cG}).$$
\begin{lemma}\label{jacobiano}
For each $n\in \mathbb{N}$, the measure $\eta_n$ is absolutely continuous with respect to $\nu$.
\end{lemma}
\begin{proof}
Consider $G$ any measurable set with diameter $\diam G<\delta/4$ and positive $\nu$-measure. Notice that if $f_{*}^{j}(\nu|_{\cG_{j}})(G)=0$ the result follows. On the other hand, if $f_{*}^{j}(\nu|_{\cG_{j}})(G)\neq 0$ then $G\subset B=B(f^j(x),\delta_0/2)$ for some $x\in \cG_j$ and thus
	$$f_{*}^{j}(\nu|_{\cG_{j}})(G)=\nu(f^{-j}(G)\cap \cG_j)=\sum_i\nu(f^{-j}_i(G\cap B)),$$
	where $f^{-j}_i:B\rightarrow V_i$ are inverse branches of $f^j$. Applying the Gibbs property and recalling that the $\nu$-measure of any open ball of radius $\delta$ is at least $\xi_\delta$, we obtain: 
	$$
	f_{*}^{j}(\nu|_{\cG_{j}})(G)\le K(\delta/2)\sum_i\frac{\nu(G)}{\nu(B)}\nu(V_i)\le K(\delta/2)\xi_\delta^{-1}\nu(G).
	$$
This proves that $\eta_n$ is absolutely continuous to $\nu$.
\end{proof}

\begin{proposition}\label{acim} There exists only one ergodic invariant measure $\mu_\varphi$ absolutely continuous with respect to $ \nu$. 
\end{proposition}
\begin{proof}
Let $\left(\mu_{n}\right)_{n}$ and $(\eta_n)_{n}$ as above defined by
$$\eta_{n}:=\frac{1}{n}\sum_{j=0}^{n-1}{f_{*}^{j}(\nu|_{\cG_{j}})}\;\; \mbox{and} \;\;\mu_{n}:=\frac{1}{n}\sum_{j=0}^{n-1}f_{*}^{j}(\nu|_{\cG}).$$
We observe that there exist  $n_{0}\in\mathbb{N}$ such that for all $n\geq n_{0}$ it holds
\begin{equation}\label{theta}
\eta_{n}(M)\,\geq\,\frac{1}{n}\sum_{j=0}^{n-1}{\nu(\cG_{j})}\,\geq\,\frac{1}{n}\sum_{j=0}^{n-1}{\nu(\cG\cap\cG_{j})}\,\geq\,\frac{\theta}{2}\nu(\cG).
\end{equation} 
In fact, by definition of positive frequency (condition~(\ref{H1})) we can consider a subset $B\subset \cG$ with $\nu(B)\geq \nu(\cG)/2$ and $n_0$ such that for every $x\in B$ and $n\geq n_{0}$ there are Gibbs times $0<n_ 1<n_2<\cdots<n_ l\leq n$ for $x$ with $l\geq \theta n.$ Hence, for every $n\geq n_0$ it follows that
$$\frac{1}{n}\sum_{j=1}^{n-1}\nu(\cG\cap\cG_{j})=\frac{1}{n}\sum_{j=1}^{n-1}\int_{\cG}1_{\cG_{j}} d\nu\geq \frac{l}{n} \nu(B)\geq \frac{\theta}{2} \nu(\cG)$$
and thus inequality~(\ref{theta}) is true.

Moreover, applying Lemma~\ref{jacobiano}, we have that $\eta_{n}$ is absolutely continuous with respect to $\nu$.
Since $M$ is a compact metric space and $f$ is continuous, we can take a subsequence $(n_{k})$  such that $\mu_{n_{k}}$ and $\eta_{n_{k}}$ converge in the weak* topology to $\mu$ and $\eta$ respectively. Notice that, by construction, $\mu$ is $f$-invariant and $\eta$ is a component of $\mu$ absolutely continuous with respect to $\nu$. Hence, we can decompose $\mu=\mu_{\varphi}+\mu_{s}$ with $\mu_{\varphi}$ absolutely continuous to $\nu$ and  $\mu_{s}$ singular to $\nu$. Since $\mu_{\varphi}(\cG)\geq\eta(\cG)>0$ we conclude the existence of some $f$-invariant probability measure $\mu_{\varphi}$ absolutely continuous to $\nu.$

Now we prove the uniqueness. Given two measures $\mu_{1}$ and $\mu_2$ absolutely continuous to $\nu$, by the theorem of ergodic decomposition, we can suppose that $\mu_{1}$ and $\mu_2$ are ergodic measures. Let $B(\mu_i)$ be the basin of attraction of $\mu_i,$ for  $i=1,2.$ Since $B(\mu_{1})$ and $B(\mu_{2})$ are disjoint invariant sets with positive $\nu$ measure, we can apply Lemma~\ref{disco} to obtain topological disks $\Delta_{i}$ of radius $\delta/4$ with $\nu(\Delta_{i}\setminus B(\mu_{i})=0$, for $i=1,2$.  From the density of pre-orbit for points in $\cG$ we have $f^{n}(\Delta_{1})\cap\Delta_{2}\neq\emptyset$ for some $n\in\mathbb{N}$. By invariance of $B(\mu_{i})$ and since %$f$ a open map and
 $\nu$ is an open measure  we obtain $\nu(B(\mu_{1})\cap B(\mu_{2}))>0$ wich implies $\mu_1=\mu_2.$
\end{proof}

In the next lemma we prove that $\mu_{\varphi}$ is equivalent to $\nu$ and thus it is also a weak Gibbs measure. In particular, $\mu_\varphi$ satisfies the property given by Corollary~\ref{gibbs_n}.

\begin{lemma} Let $\mu_\varphi$ be the unique ergodic invariant measure absolutely continuous with respect to $\nu$. Then the density $d\mu_{\varphi}/d\nu$ is bounded away from zero and infinity. In particular, $\mu_{\varphi}$ is a weak Gibbs measure.
\end{lemma}
\begin{proof}  Since $\mu_{\varphi}$ is a $f$-invariant measure absolutely continuous to $\nu$, the density $h:=d\mu_{\varphi}/d\nu$ is a $L^{1}(\nu)$ function and it satisfies: 
\begin{eqnarray*} 
\int{\psi\mathcal{L}_{\varphi}(h)}d\nu&=&\int{\mathcal{L}_{\varphi}(\psi\circ f\cdot h)}d\nu=\int{\psi\circ f\cdot h}d\mathcal{L}^{*}_{\varphi}(\nu)
=\lambda\!\int{\psi\circ f\cdot h}d\nu\\
&=&\lambda\int{\psi\circ f}\,d\mu=\lambda\int{\psi}\,d\mu=\int{\psi\cdot \lambda h}\,d\nu\\
\end{eqnarray*} 
for every continuous function $\psi.$ Thus, $\mathcal{L}_{\varphi}(h)=\lambda h$ in $\nu$-almost every where. Recalling the construction of $\mu_\varphi$ we have that $h=d\mu_{\varphi}/d\nu\geq d\eta/d\nu$  where $\eta$ is an accumulation point of the sequence $$\eta_{n}:=\frac{1}{n}\sum_{j=0}^{n-1}{f_{*}^{j}(\nu|_{\cG_{j}})}.$$
We will show that  $d\eta_{n}/d\nu$ is uniformly bounded away from zero. 

Fix some $\varepsilon<\delta_0$. Given $y\in\cG$, by the density of its pre-image, there exists some $s$ depending only $\varepsilon$ such that we can cover each $\cG_{j}$ by dynamical balls $\{B_{\varepsilon}(f^{-j}(z), j)\,/\;z\in f^{-s}(y)\}.$ Thus, for every $n>0$ we have 
\begin{eqnarray*}
\sum_{z\in f^{-s}(y)}{\frac{d\eta_{n}}{d\nu}(z)}&=&\sum_{z\in f^{-s}(y)}\frac{1}{n}\sum_{j=0}^{n-1}{\frac{df_{*}(\nu|_{\cG_{j}})}{d\nu}(z)}\\\\
&=&\sum_{z\in f^{-s}(y)}\frac{1}{n}\sum_{j=0}^{n-1}\big\{\sum_{w\in f^{-j}(z)}{\lambda^{-j}e^{S_{j}\varphi(w)}}\big\}\\\\
&\geq&\sum_{z\in f^{-s}(y)}\frac{1}{n}\sum_{j=0}^{n-1}\big\{\sum_{w\in f^{-j}(z)}\!\!\!\!K^{-1}(\vep)\nu(B_{\varepsilon}(w,j))\big\}\\\\
&\geq& K^{-1}(\vep)\frac{1}{n}\sum_{j=0}^{n-1}{\nu(\cG_{j})}\\\\
&\geq& K^{-1}(\vep)\theta.
\end{eqnarray*} 
As $\mathcal{L}_{\varphi}(h)=\lambda h$ we conclude that 
\begin{eqnarray*}
h(y)=\lambda^{-s}\sum_{z\in f^{-s}(y)}{e^{S_{s}\varphi(z)}h(z)}&\geq& \lambda^{-s}\sum_{z\in f^{-s}(y)}{e^{S_{s}\varphi(z)}\frac{d\eta}{d\nu}(z)}\\ \\
&\geq& \lambda^{-s}e^{s\inf \varphi}K^{-1}(\vep)\theta>0.
\end{eqnarray*}
This proves that $\mu_\varphi$ is equivalent to $\nu$: there exists $K_1=K_1(\vep)>0$ s.t.
$${K_{1}}^{-1}\nu(B_\varepsilon(x, n))\leq \mu_{\varphi}(B_\varepsilon(x, n))\leq {K_ 1}\nu(B_\varepsilon(x, n)) $$
for $\nu$-almost every $x\in M$, every $\varepsilon>0$ and every $n\geq 1$. In particular, applying Lemma~\ref{conforme}, we conclude that $\mu_\varphi$ is a weak Gibbs measure: for $n$ a Gibbs time for $x\in \cG$ we have 
\begin{eqnarray*}
 K^{-1} \leq  \frac{\mu_{\varphi}(B_\varepsilon(x, n))}{\exp(S_n\varphi(y) - nP_{f}(\varphi))}  \leq K
\end{eqnarray*}
for all $y\in B_\varepsilon(x, n)$. By a slight abuse of notation, we also use $K$ to denote the Gibbs constant of $\mu_\varphi$.
\end{proof}

\subsection{Equilibrium state for local homeomorphisms.} Now, we will prove Theorem~\ref{ee base}. First, we observe that the weak Gibbs measure $\mu_{\varphi}$ of Theorem~\ref{Gibbs base} is an equilibrium state for $(f, \varphi)$: given $x \in \cG$, let $(n_i)$ be a sequence of Gibbs times. Since $\mu_{\varphi}$ is ergodic, we can apply Brin-Katok's local entropy formula and the Gibbs property to obtain 
	$$	h_{\mu_{\varphi}}(f)=h_{\mu_{\varphi}}(f, x) = -\lim \frac{1}{n_i}\log\mu_{\varphi}(B_{\varepsilon}(x, n_ i)) = P_{f}(\varphi)-\int \varphi\,d\mu_{\varphi}.
	$$  
This implies that $\mu_{\varphi}$ is an equilibrium state.

In the next result, we prove that there is no other equilibrium state associated to $(f, \varphi).$ To this end, we assume the existence of a generating partition $\tau$ for $\cG$; that is, 
$\tau$ is a partition of $\cG$ with finite entropy such that its iterates, defined by $\tau^n=\vee_{j=0}^{n-1}f^{-j}\tau$ for $n\geq 1$, generate the Borel $\sigma$-algebra of $\cG.$
\begin{theorem} Let $f:M\to M$ be a local homeomorphism satisfying~(\ref{H1}) and let ${\varphi}: M \to \mathbb{R}$ be a continuous potential satisfying~(\ref{H2}) and (\ref{H3}). Suppose that $\cG$ admits a generating partition. Then the unique weak Gibbs measure $\mu_{\varphi}$ is the unique equilibrium state of $(f, \varphi).$
\end{theorem}

\begin{proof} Let $\eta$ be an ergodic equilibrium state for $(f, \varphi)$, that means, 
$$\log\lambda=P_{f}(\varphi)=h_{\eta}(f)+\int{\varphi}\,d\eta.$$

Let $\tau$ be a generating partition of $\cG$ where $f$ is injective on each atom of $\tau$. Notice that $\tau$ is a generating partition with respect to $\eta$ because the condition~(\ref{H3}) on the potential, i.e.,
\begin{equation*}
P_{f}(\varphi, \cG^{c})<P_{f}(\varphi, \cG)=P_{f}(\varphi)
\end{equation*} implies that $\eta(\cG)=1$, for any equilibrium state of $(f, \varphi).$ 

Define the partition $\xi$ by $\xi:=\bigvee_{n\geq 0}f^{n}\tau$. Observe that $f^{-1}\xi$ is finer than $\xi$ and $\bigvee_{n\geq 0}f^{-n}\xi$ is also a generating partition. We will compare $\mu_{\varphi}$ and $\eta$ with respect to this partition $\xi.$

Denote by $(\eta_{x})_{x}$ and by $(\mu_{x})_{x}$ the Rokhlin's disintegration of the measures $\eta$ and $\mu_{\varphi}$ on the partition $\xi$, respectively.  
Recalling that $\mu_{\varphi}=h\nu$, its decomposition $(\mu_x)_{x}$ over $\xi$ is given by
$$\mu_{x}(A)=\frac{1}{\mu_{\varphi}(\xi(x))}\displaystyle\int_{A\cap\xi(x)}{h_{x}(y)}\,d\nu_x(y),$$ for any Borel set $A\subset M.$ Since $f^{-1}\xi$ is finer than $\xi$ we have
\begin{eqnarray*}
\mu_{x}(f^{-n}\xi(x))&=&\frac{1}{\mu_{\varphi}(\xi(x))}\int_{f^{-n}\xi(x)}{h_{x}(f^{-n}(y))}\,d\nu_{x}(f^{-n}(y)),
\end{eqnarray*} for every $n\in \mathbb{N}.$ Using change of coordinates, we compute 
\begin{eqnarray*}
\mu_{x}(f^{-n}\xi(x))&=&\frac{1}{\mu_{\varphi}(\xi(x))}\int_{f^{-n}\xi(x)}{h_{x}(f^{-n}(y))}\,d\nu_{x}(f^{-n}(y))\\\\
&=&\frac{1}{\mu_{\varphi}(\xi(x))}\int_{\xi(x)}{h_{x}(f^{-n}(y))\lambda^{-n}e^{S_{n}{\varphi}(x)}}\,d\nu_{x}(f^{-n}(y))\\\\
&=&\frac{1}{\mu_{\varphi}(\xi(x))}\int_{f^{n}\xi(x)}{\frac{h_{x}(f^{-n}(y))\lambda^{-n}e^{S_{n}{\varphi}(y)}}{\lambda^{n}e^{-S_{n}{\varphi}(x)}}}\,d\nu_{f^{n}(x)}(y)\\\\
&=&\frac{1}{\mu_{\varphi}(\xi(x))}\int_{f^{n}\xi(x)}{\frac{h_{f^{n}(x)}(y)}{\lambda^{n}e^{-S_{n}{\varphi}(x)}}}\,d\nu_{f^{n}(x)}(y)\\\\
&=&\frac{\mu_{\varphi}(f^{n}\xi(x))}{\mu_{\varphi}(\xi(x))J_{\nu}f^{n}(x)}\\
\end{eqnarray*}
on the fourth equality we use  $\mathcal{L}_{\varphi}(h)=\lambda h$ in almost every where. Then 
%Note that, $\log\mu_{x}(f^{-n}\xi(x))$ is a non-positive function and therefore it has an $\eta$-integrable positive part. 
$$\log\mu_{x}(f^{-n}(\xi)(x))=\log \mu_{\varphi}(f^{n}\xi(x))-\log\mu_{\varphi}(\xi(x))-\log J_{\nu}f^{n}(x).$$ 
By Birkhoff’s Ergodic Theorem, we obtain
\begin{eqnarray*}
-\int{\log\mu_{x}(f^{-n}(\xi)(x))}\,d\eta(x)&=&\int{\log J_{\nu}f^{n}(x)}\,d\eta(x)\\
&=&\int{\log \lambda^{n}e^{-S_{n}\phi(x)}}\,d\eta(x)\\
&=&n\,(\log\lambda-\int{\phi}\,d\eta)\\
&=&n h_{\eta}(f).
\end{eqnarray*}
On the other hand, since $\xi$ is a generating partition, it follows that
$$n h_{\eta}(f)=n h_{\eta}(f,\xi)=H_{\eta}(f^{-n}\xi|\xi)=-\int{\log\eta_{x}(f^{-n}(\xi)(x))}\,d\eta(x).$$
Finally, applying the concavity of the logarithm, we obtain the equality
\begin{eqnarray*}
0=-\int{\log\frac{\mu_{x}(f^{-n}(\xi)(x))}{\eta_{x}(f^{-n}(\xi)(x))}}\,d\eta(x)&=&\int{\log\frac{d\mu_{\phi}}{d\eta}}\,d\eta\\
&\leq&\log\int{\frac{d\mu_{\phi}}{d\eta}}\,d\eta=0\\
\end{eqnarray*}
which implies $\eta=\mu_{\varphi}$ on the $\sigma$-algebra generated by $f^{-n}(\xi)$. As $f^{-n}(\xi)$ goes to the total $\sigma$-algebra when $n\rightarrow\infty$, we prove that $\eta=\mu_{\varphi}$ and this finishes the proof.
\end{proof}

%%%%%%%%%%%%%%%%%%%%%%%%%%%%%%%%%%%%%%%%%%%%%

%%%%%%%%%%%%%%%%%%%%%%%%%%%%%%%%%%%%%%%
\section{Liftability of the weak Gibbs measure}
\label{atrator}
In this section we study the class of homeomorphisms $F:N \to N$  that are semiconjugated to local homeomorphisms $f:M\to M$ as defined in Section~\ref{results}.  We will prove Theorems ~\ref{theo.Gibbs} and~\ref{theo finitude} by applying the results of Theorems~\ref{Gibbs base} and~\ref{ee base}. Our strategy is to show that the weak Gibbs measure of $f$ is liftable to a weak Gibbs measure for the attractor. In \cite{FO18}, the first author, together with Oliveira, addressed the inverse problem of preserving the Gibbs property under projection. 

\subsection{Cohomologous potential} The next result shows that every H\"older continuous potential which satisfies conditions~(\ref{H4}) and (\ref{H5}) is cohomologous to one which does not depend on the stable direction and preserves the properties~(\ref{H4}) and (\ref{H5}).

\begin{proposition}\label{homologo} Let $\phi: N \to \mathbb{R}$ be a H\"older continuous potential satisfying~(\ref{H4}) and (\ref{H5}). Then there exists a continuous potential $\bar{\phi}:  N \to \mathbb{R}$ cohomologous to $\phi$ which does not depend on the stable direction. Moreover, $\bar{\phi}$ also satisfies the conditions~(\ref{H4}) and (\ref{H5}).
\end{proposition}
 \begin{proof} Let $r:N\rightarrow N$ be a continuous function that not depending on the stable direction:  for each $\hat{x}\in N$ the image $r(\hat{x})$ is a point in the local stable manifold $N_{\hat{x}}$ and $r(\hat{x})=r(\hat{y})$ for all $\hat{y}\in N_{\hat{x}}$. Define $u:N\rightarrow\mathbb{R}$ by $$u(\hat{x})=\sum_{j=0}^{\infty}\phi\circ F^{j}(\hat{x})-\phi\circ F^{j}(r(\hat{x})).$$
Since $\hat{x}$ and $r(\hat{x})$ are in the same stable direction $N_{\hat{x}}$, for all $j>0$ we have   
$$ d_N(F^{j}(\hat{x}), F^{j}( r(\hat{x})) \leq \lambda^j d_{N}(\hat{x}, r(\hat{x})),$$
where $\lambda$ is the contraction rate of $F$ . Since $\phi$ is a H\"older function we obtain
\begin{eqnarray*}
u(\hat{x})=\sum_{j=0}^{\infty}\phi\circ F^{j}(\hat{x})-\phi\circ F^{j}(r(\hat{x}))\leq c\sum_{j=0}^{\infty} d_N(F^{j}(\hat{x}), F^{j}(r(\hat{x})))\leq \tilde{c} \sum_{j=0}^{\infty} \lambda^{ j}.
\end{eqnarray*}
Therefore $u$ is well defined and it is a continuous function.

Now, consider the potential $\bar{\phi}:N\rightarrow\mathbb{R}$\; defined by\; $\bar{\phi}:=\phi-u+u\circ F$. Then $\bar{\phi}$ is a continuous potential cohomologous to $\phi$ and it holds that
\begin{eqnarray*}
\bar{\phi}&=&\phi-u+u\circ F\\ 
&=&\phi-\sum_{j=0}^{\infty}\left(\phi\circ F^{j}-\phi\circ F^{j}\circ r\right)+\sum_{j=0}^{\infty}\left(\phi\circ F^{j+1}-\phi\circ F^{j}\circ r\circ F\right)\\\\
&=&\phi\circ r+\sum_{j=0}^{\infty}\left(\phi\circ F^{j+1}\circ r-\phi\circ F^{j}\circ r\circ F\right).
\end{eqnarray*}
Therefore, $\bar{\phi}$ is constant on the stable direction. Moreover, since 
$$|S_{n}\bar\phi(\hat{y})-S_{n}\bar\phi(\hat{z})|\leq |S_n\phi(\hat{y})-S_n\phi(\hat{z})|+2\|u\|,$$
we conclude that $\bar{\phi}$ satisfies~(\ref{H4}) and, by Proposition~\ref{pressao cohomologo}, $\bar{\phi}$ satisfies~(\ref{H5}).
\end{proof}

According to Proposition~\ref{homologo} for any $\phi:  N \to \mathbb{R}$ H\"older continuous potential for~$F$,  there exists a continuous potential $\bar{\phi}:  N \to \mathbb{R}$ that is cohomologous to $\phi$ and not depending on the stable direction. Hence, the potential $\bar{\phi}$ induces a continuous potential $\varphi:M\rightarrow\mathbb{R}$ given by 
\begin{equation}\label{eq.phis}
 \bar \phi =\varphi\circ\pi,
 \end{equation}
%\begin{equation}\varphi(x)=\bar{\phi}(x, z).\label{eq.fizinho}
%\end{equation} 
where $\pi: N\rightarrow M$ is the continuous surjection which conjugates $F$ and $f$. 
%Moreover, using equation... it is straightforward to check that
%\begin{equation}\label{eq.pressoes}
%P_{f}(\varphi)\leq P_{F}(\bar{\phi}).
%\end{equation}
%Now we are going to prove that there exists a bijection between the sets of equilibrium states for $(F,  {\phi})$ and for $(f, \varphi)$. 
In the next result, we prove that $\varphi$ satisfies the conditions~(\ref{H2}) and (\ref{H3}). As a consequence, we can apply our Theorem~\ref{Gibbs base} for $(f, \varphi).$
% $$P_{f}(\varphi, (\cG)^c)
% <  P_f(\varphi, \cG)=P_f(\varphi).$$
\begin{lemma}\label{le.fizinho}
 If $\phi: N \to \mathbb{R}$ is a H\"older continuous potential satisfying~(\ref{H4}) and~(\ref{H5}) then $\varphi:M\to\mathbb{R}$ is a continuous potential satisfying~(\ref{H2}) and~(\ref{H3}).
 \end{lemma}

 \begin{proof} We start by showing the condition~(\ref{H3}) that is 
$ P_{f}(\varphi, \cG^c)
 < P_f(\varphi, \cG).$
Given $\vep>0$ and $x\in M$, we can consider $\hat{\vep}>0$ and $\hat{x}\in N$ such that $\pi(\hat{x})=x$ and $B(x, \vep)\subset \pi(B(\hat{x}, \hat{\vep}))$ where $\hat{\vep}\to 0$ when $\vep\to 0.$  For $n\in\mathbb{N}$, we denote by $B^F_{\hat{\vep}}(\hat{x},n)$ and $B^f_{\vep}(x,n)$ the dynamical ball of $\hat{x}$ and $x$ respectively, where the superscripts indicate the corresponding dynamics.
Notice that the existence of holonomies $h_{x,y}:  N_x \to N_y$ and a constant $C\geq 1$ satisfying
$$\frac{1}{C}[d_M (x, y) + d_N(h_{x,y}(\hat{x}), \hat{y})]\leq d_N(\hat{x}, \hat{y}) \leq  C [d_M (x, y) + d_N(h_{x,y}(\hat{x}), \hat{y})]$$
implies that
\begin{equation}
\pi\left(B^F_{C^{-1}\vep}(\hat{x}, n)\right)\subset B^f_{\vep}(x, n)\subset \pi\left(B^F_{C\tilde{\vep}}(\hat{x}, n)\right), \label{inclusao}
\end{equation}
where $\tilde{\vep}\to 0$ as $\vep\to 0$. Indeed, from
\begin{equation*}
    d_M(f^j(x), f^j({y}))\leq C d_N(F^j(\hat{x}), F^j{(\hat{y}}))
\end{equation*}
we have $\pi\left(B^F_{C^{-1}\vep}(\hat{x}, n)\right)\subset B^f_{\vep}(x, n).$ Moreover, as $F$ contracts on fibers then
\begin{eqnarray*}
    d_N(F^j(\hat{x}), F^j({\hat{y}}) )&\leq& C [d_M(f^j(x), f^j(y)) + d_N( F^j(h_{x, y} (\hat{x})),   F^j(\hat{y})) ]  \\
    &\leq& C[\vep+\lambda^j d_{N}(h_{x, y}(\hat{x}), \hat{y})]\\
    &\leq&C[\vep+\lambda^j C\hat{\vep}] \\
    &\leq& C\tilde{\vep}. 
\end{eqnarray*}
for every $0\leq j\leq n$, where $\tilde{\vep}\to 0$ as $\vep\to 0$. Thus, $B^f_{\vep}(x, n)\subset \pi\left(B^F_{C\tilde{\vep}}(\hat{x}, n)\right)$. 
From (\ref{inclusao}) it follows that
\begin{eqnarray*}
R_{n, C^{-1}\vep}\bar\phi(\hat{x})&=&   \sup_{\hat{y}\in B^F_{C^{-1}\vep}(\hat{x},n)} S_n \bar\phi(\hat{y}) \\ &=& \sup_{\hat{y}\in B^F_{C^{-1}\vep}(\hat{x},n)}S_n (\varphi\circ\pi)(\hat{y}) \\
&=& \sup_{y\in \pi (B^F_{C^{-1}\vep}(\hat{x},n))}S_n \varphi (y) \\ & \leq & \sup_{y \in B^f_{\vep}(x,n)} S_n \varphi(y) \\
&=&R_{n,\vep} \varphi(x)\\
&\leq& \sup_{y\in \pi (B^F_{C\tilde{\vep}}(\hat{x},n))}S_n \varphi (y) \\
&=& \sup_{\hat{y}\in B^F_{C\tilde{\vep}}(\hat{x},n)}S_n (\varphi\circ\pi)(\hat{y})\\
&=& R_{n, C\tilde{\vep}}\bar\phi(\hat{x}),
\end{eqnarray*}
i.e. 
\begin{eqnarray}
R_{n, C^{-1}\vep}\bar\phi(\hat{x})\leq R_{n,\vep} \varphi(x) \leq R_{n, C\tilde{\vep}}\bar\phi(\hat{x}). \label{R}
\end{eqnarray}
Therefore, for any $\Lambda\subset M$ and $\hat{\Lambda}=\pi^{-1}(\Lambda)\subset N$, positively invariant sets, the inequality~(\ref{R}) yields 
\begin{eqnarray*}
m_F(\bar\phi, \hat{\Lambda}, C^{-1}\vep, n, \gamma)
 &\leq& m_f(\varphi, \Lambda, \vep, n, \gamma)\\
 &\leq& m_F(\bar\phi, \hat{\Lambda}, C\tilde{\vep}, n, \gamma).
\end{eqnarray*}
As $n$ goes to infinity we obtain
\begin{eqnarray*}
m_F(\bar\phi, \hat{\Lambda}, C^{-1}\vep, \gamma)
 \leq m_f(\varphi, \Lambda, \vep, \gamma)
 \leq m_F(\bar\phi, \hat{\Lambda}, C\tilde{\vep},  \gamma).
\end{eqnarray*}
Taking the infimum on $\gamma$ we have 
$$P_{F}(\bar{\phi}, \hat{\Lambda}, C^{-1}\vep)\leq P_{f}(\varphi, \Lambda, \vep) \leq P_{F}(\bar{\phi}, \hat{\Lambda}, C\tilde{\vep}) $$
As $\vep\to 0$ follows that  $\tilde{\vep}\to 0$ and thus 
$$P_{f}(\varphi, \Lambda)=P_{F}(\bar{\phi}, \hat{\Lambda}). $$
%Since $B^F_{C^{-1}\delta}(\hat{x}, n)\subset B^F_{\delta}(\hat{x}, n)\subset B^F_{C\delta}(\hat{x}, n)$ and the quantity $m_F$ is \textcolor{red}{monotone non-decreasing as $\delta$ decrease} it follows that
%\begin{eqnarray*} m_F(\bar\phi, (\hat{\Sigma}(F)))^c, \delta, N, \gamma) &\leq& m_F(\bar\phi, (\hat{\Sigma}(F)))^c, C^{-1}\delta, N, \gamma)\\
%& \leq& m_f(\varphi, (\Sigma(f))^c, \delta, N, \gamma)\\
%&\leq& m_F(\bar\phi, (\hat{\Sigma}(F)))^c, \tilde{C}\delta, N, \gamma) \\
%&\leq& m_F(\bar\phi, (\hat{\Sigma}(F)))^c, \delta, N, \gamma). 
%\end{eqnarray*}
%Thus,  for each $N\in \mathbb{N}$ and $\gamma>0$ we have
%$$  m_f(\varphi, (\Sigma(f))^c, \delta, N, \gamma) =    %m_F(\bar\phi, (\hat{\Sigma}(F))^c , \delta, N, \gamma)$$ 
%and for the same reason
%$$  m_f(\varphi, \Sigma(f), \delta, N, \gamma) =    m_F(\bar\phi, \hat{\Sigma}(F), \delta, N, \gamma).$$ 
%This implies that $$P_{f}(\varphi, \left(\Sigma(f)\right)^c)= P_{F}(\bar{\phi}, (\hat{\Sigma}(F))^{c})$$
Similarly, $$P_{f}(\varphi, \Lambda^c)=P_{F}(\bar{\phi}, \hat{\Lambda}^c).$$
In particular, 
\begin{equation}\label{mesma pressao}
P_{f}(\varphi)=P_{F}(\bar{\phi}).
\end{equation}
Since $\bar\phi=\varphi \circ \pi$ satisfies~(\ref{H5}), by Proposition~\ref{homologo}, we obtain for $\Lambda=\cG$ that
$$P_{f}(\varphi, \cG^c)
= P_{F}(\bar{\phi},\hat{\cG}^{c}) < P_F(\bar{\phi}, \hat{\cG}) = P_f(\varphi, \cG).$$
Then $\varphi$ satisfies condition~(\ref{H3}). 

Now we prove that $\varphi$ satisfies condition~(\ref{H2}). Let $B^f_{\vep}(x, n)$ be a dynamic ball of $x\in M$ where $n$ is a Gibbs time. Given $y, z \in B^f_{\vep}(x, n)$, let $\hat{y}, \hat{z}\in B_{C\tilde{\vep}}(\hat{x}, n)$ be such that $\pi(\hat{y})=y, \pi(\hat{z})=z$. 
%Since $n$ is a hyperbolic time for $x$ and $F$ contracts on fibers  we conclude that
%\begin{eqnarray}
%%    d_N(F^j(\hat{y}), F^j({\hat{z}}) )&\leq& C [d_M(f^j(y), f^j(z)) + d_N( F^j(h_{y, z} (\hat{y})),   F^j(\hat{z})) ]  \nonumber \\ 
%    &\leq& C[\sigma^{n-j}d_M(y, z)+\lambda^jd_{N}(h_{y, z}(\hat{y}), \hat{z})] \nonumber \\
%    &\leq& C[\sigma^{n-j}d_M(y, z)+\lambda^j\textcolor{red}{\mbox{diam}(N)}]\label{Bowen}
%\end{eqnarray}
%for $0\leq j \leq n$. 
Recalling that $\bar\phi=\varphi\circ\pi$ and $\bar\phi=\phi-u+u\circ F$ 
%for some choose $u_m=u$. For $n_k$ Gibbs time of $y$ and $z\in B(x,n_k,\vep)$ 
we have
\begin{align*}
|S_{n}\varphi(y)-S_{n}\varphi(z)|=|S_{n}\bar\phi(\hat{y})-S_{n}\bar\phi(\hat{z})|\leq |S_n\phi(\hat{y})-S_n\phi(\hat{z})|+2\|u\|.
 %\log \bar{K}_{n_k}(y,\vep)
\end{align*}
Since $\phi$ satisfies condition~(\ref{H4}) we conclude that $|S_n\varphi(y)-S_n\varphi(z)|\leq \tilde{V}$ for all $y, z\in B^f_{\vep}(x, n)$ and $n$ Gibbs time of $x\in M$. 
%where $K=\sup_{\hat{x}\in N} \|u\|$ and $\bar{K}_{n_k}(y,\vep)=\exp(\log K_{n_k}(y,\vep)+2K)$.
\end{proof} 

%From the $\alpha$-H\"older continuity of $\phi$ and from inequality~(\ref{Bowen}) we obtain 
%\begin{eqnarray*}
%|S_n\varphi(y)-S_n\varphi(z)|&\leq& |S_n\phi(\hat{y})-S_n\phi(\hat{z})|+2K\\
%&\leq&(\sum_{j=0}^{n}|\phi\circ F^{j}(\hat{y})-\phi\circ F^{j}(\hat{z})|)+2K\\
%&\leq& c\,(\sum_{j=0}^{n} d(F^{j}(\hat{y}), F^{j}(\hat{z}))^{\alpha})+2K\\
%%&\leq& c\,(\sum_{j=0}^{n} [C(\sigma^{n-j}\delta + \lambda^j\mbox{diam}(N))\}]^{\alpha}) + 2K\\
%&\leq& c\,\tilde{C}^{\alpha}\delta^{\alpha}\, (\sum_{j=0}^{+\infty} \max{\{\sigma, \lambda\}}^{j\beta}) +2K
%\end{eqnarray*}

\subsection{Gibbs measure for the attractor} 
%As in Proposition \ref{homologo}, given a H\"older continuous potential $\phi: N\rightarrow \mathbb{R}$ we can consider $\bar{\phi}: N\rightarrow \mathbb{R}$ a continuous potential that does not depend on the stable direction. Therefore, we can define the continuous potential $\varphi: M \rightarrow \R$ by $\bar\phi=\varphi\circ \pi$, where $\pi$ is the semiconjugacy between $F$ and $f$. Since $\varphi$ is a continuous potential satisfying~(\ref{H2}) and~(\ref{H3}), 

Let $\mu_{\varphi}$ be the ergodic weak Gibbs measure associated to $(f, \varphi)$ given by Theorem~\ref{Gibbs base}. In the next result, we show the uniqueness of a $F$-invariant probability measure $\mu_{\phi}\in \mathbb{P}_{F}(N)$ that projects on $\mu_{\varphi}.$

 \begin{lemma} There exists only one ergodic measure $\mu_{\phi}$ such that $\mu_\varphi=\pi_{*}\mu_\phi.$
\end{lemma} 
\begin{proof} The existence of $\mu_{\phi}\in \mathbb{P}_{F}(N)$ satisfying $\mu_\varphi=\pi_{*}\mu_\phi$ is a classical result whose proof can be found in \cite[Lema 5.3.]{ARS18}.
We are going to show the uniqueness of $\mu_{\phi}$. Let ${\mu}_1, {\mu}_2\in \mathbb{P}_{F}(N)$ be ergodic measures such that $$\pi_{*}{\mu}_1=\mu_\varphi=\pi_{\ast}{\mu}_2.$$ Consider
 $B_{{\mu}_1}(F)$ and $B_{{\mu}_2}(F)$ the basins of attraction of ${\mu}_1$ and ${\mu}_2$, respectively. From the uniform contraction of $F$ on fibers and the ergodicity of ${\mu}_1$ and ${\mu}_2$, there are Borel sets $A_1,A_2\subset M$ with $A_{1}\cap A_{2}=\emptyset$ such that
$$B_{{\mu}_1}(F)=\bigcup_{y\in A_{1}}N_y\quad\text{and}\quad B_{{\mu}_2}(F)=\bigcup_{z\in A_{2}}N_z .$$  
Moreover, since $\mu_\varphi$ is ergodic and the sets $\pi(B_{{\mu}_1}(F))$, $\pi(B_{{\mu}_2}(F))$ are invariant by $f$, we have that 
$$\mu_{\varphi}\big(\pi(B_{{\mu}_1}(F))\big)=\mu_{\varphi}\big(\pi(B_{{\mu}_2}(F)\big)\big)=1.$$
Hence 
 $$\pi\big(B_{{\mu}_1}(F)\big)\cap \pi\big(B_{{\mu}_2}(F)\big)\neq\emptyset$$
and by definition of $\pi$ it follows that $A_{1}\cap A_{2}\neq\emptyset$ and thus, ${\mu}_1={\mu}_2$.

\end{proof}

In what follows, we establish that the Gibbs property is preserved under projection and liftability.
\begin{theorem} 
$\mu_{\varphi}$ is a weak Gibbs measure for $(f,\varphi)$ if, and only if, $\mu_{\phi}$ given by $\pi_*\mu_{\phi}=\mu_{\varphi}$ is a weak Gibbs measure for $(F,\phi)$. 
\end{theorem}

\begin{proof}
	Given $\hat{x}, \hat{y} \in N$ consider $x = \pi(\hat{x}), y=\pi(\hat{y})$ in $M$ and
	recall the inclusion~(\ref{inclusao}) on the proof of Lemma \ref{le.fizinho},  that is,
\begin{equation} \label{inclusion}
\pi\left(B^F_{C^{-1}\vep}(\hat{x}, n)\right)\subset B^f_{\vep}(x, n)\subset \pi\left(B^F_{C\tilde{\vep}}(\hat{x}, n)\right),
\end{equation}
where $B^F_\vep(\hat{x}, n)$, $B^f_{\vep}(x,n)$ are the dynamic balls of $F$ and $f$, respectively. Moreover, $\tilde{\vep}=\vep+C\hat{\vep}$ and $\hat{\vep}\to 0$ as $\vep\to 0$. 

We recall also that, for $n$ sufficiently large, we have 
\begin{equation}\label{boladinamica}
\pi^{-1}(B_{\vep}^{f}(x, n))= B_{C\tilde{\vep}}^{F}(\hat{x}, n)\cap \pi^{-1}(B_{\vep}^{f}(x, n))
\end{equation}

%\marginpar{essa igualdade 7.6 me parece um pouco estranha, por conta dos raios q variam...acho q escrever em termos da medida parece mais razoavel}
% \begin{equation}\label{boladinamica}
% {\color{red}\mu_{\phi}(\pi^{-1}(B_{\vep}^{f}(x, n)))= \mu(B_{C\tilde{\vep}}^{F}(\hat{x}, n)\cap \pi^{-1}(B_{\vep}^{f}(x, n)))}
% \end{equation}

because
\begin{eqnarray*}
    d_N(F^j(\hat{x}), F^j({\hat{y}}) )&\leq& C [d_M(f^j(x), f^j(y)) + d_N( F^j(h_{x, y} (\hat{x})),   F^j(\hat{y})) ]  \\
    &\leq& C[\vep+\lambda^j d_{N}(h_{x, y}(\hat{x}), \hat{y})]
    %&\leq&C[\vep+\lambda^j C\hat{\vep}].   
\end{eqnarray*}
Therefore,  from~(\ref{inclusion}) and~(\ref{boladinamica}), it follows that  
	\begin{eqnarray*}
\displaystyle	\frac{\mu_{\phi}\left(B_{C^{-1}\vep}^F(\hat{x},n)\right)}{\exp(S_{n}\phi(\hat{y})-n\,P_{F}(\phi))}
		&\le& \frac{\mu_{\phi}\left(\pi^{-1}(B_{\vep}^f(x,n))\right)}{\exp(S_{n} \bar{\phi}(\hat{y})+u(\hat{y})-u\circ F^n(\hat{y})-n\,P_{F}(\phi))}\\  
		&=&\frac{\mu_{\varphi}(B_{\vep}^f(x,n))}{\exp(S_{n}\varphi(y)+u(\hat{y})-u\circ F^n(\hat{y})-n\,P_{f}(\varphi))}\\
        &=&\frac{\mu_{\phi}\left(\pi^{-1}(B_{\vep}^f(x,n))\right)}{\exp(S_{n}\phi(\hat{y})-n\,P_{F}(\phi))} \\
                &\leq& \frac{\mu_{\phi}\left(B_{C\tilde{\vep}}^F(\hat{x},n)\right)}{\exp(S_{n}\phi(\hat{y})-n\,P_{F}(\phi))}.\\
	\end{eqnarray*}	

Since $u$ is a bounded function and  $\vep>0$ is arbitrary, it follows that $\mu_{\varphi}$ is a weak Gibbs measure for $(f,\varphi)$ if, and only if, $\mu_{\phi}$ also is for $(F,\phi)$.
\end{proof}

The uniqueness of the weak Gibbs measure $\mu_{\varphi}$ for the base dynamic $(f, \varphi)$
was obtained in Theorem~\ref{Gibbs base}; Thus, by applying the last result, we obtain the uniqueness of the weak Gibbs measure $\mu_{\phi}$ for $(F, \phi)$. This proves Theorem~\ref{theo.Gibbs}.

\subsection{Equilibrium states for the attractor.} 
%In this section we prove the existence of finitely many ergodic equilibrium states for the partially hyperbolic attractor with respect to hyperbolic H\"older continuous potentials, namely Theorem~\ref{theo finitude}.
%We point out that Theorem~\ref{finitude} enlarge the class of potentials considered  in \cite{RamosViana2017}, where it is also required that the potential does not depend on the stable direction. Here we prove that this condition is not necessary.

Now we can prove Theorem~\ref{theo finitude}. The strategy is to show the existence of a bijection between equilibrium states of $(F, \phi)$ and $(f, \varphi).$ Therefore, Theorem~\ref{theo finitude} will follow from the uniqueness given by Theorem~\ref{ee base}. For the bijection, we will use the Ledrappier and Walter Formula which relates the topological entropy of semiconjugated systems; see \cite{LW} for a proof of this result. 

%We start by showing that the Gibbs measure $\mu_{\phi}$ of Theorem~\ref{theo.Gibbs} is an equilibrium state for $(F, \phi)$. In fact, since $\mu_{\phi}$ is ergodic, we can apply Brin-Katok's local entropy formula and the Gibbs property to prove that for almost every  $x \in \hat{\Sigma}$, if $n_i(x)$ is the sequence of Gibbs times of $x$, then   
%	$$ 
	%h_{\mu_{\phi}}(F)=h_{\mu_{\phi}}(F, x) = -\lim \frac{1}{n_i}\log\mu_{\phi}(B(x,n_i,\vep)) = P_{F}(\phi)-\int \phi\,d\mu_{\phi}$$  
%which implies that $\mu_{\phi}$ is an equilibrium state. 

%Since $F$ contracts on fibers, we can prove the existence of a bijection between equilibrium states of $(F, \phi)$ and $(f, \varphi).$ 

\begin{theorem}[Ledrappier-Walter's Formula] Let $N$, $M$ be compact metric spaces and let $F:N\rightarrow N$, $f:M\rightarrow M$, $\pi:N\rightarrow M$ be continuous maps such that $\pi$ is surjective and $\pi\circ F=f\circ\pi$ then
$$\sup_{\tilde\mu;\pi_{*}\tilde{\mu}=\mu}h_{\tilde{\mu}}(F)=h_{\mu}(f)+\int{h_{top}(F, \pi^{-1}(y))\,d\mu(y)}$$
\end{theorem}

Once again, as in Proposition~\ref{homologo}, given a  H\"older continuous potential $\phi: N\to \mathbb{R}$ satisfying~(\ref{H4}) and~(\ref{H5}), let  $\bar\phi: N\to \mathbb{R}$ be the cohomologous potential and let $\varphi: M\to \mathbb{R}$ be a continuous potential for $f$ satisfying~(\ref{H2}) and~(\ref{H3}) induced by $\bar\phi$. 

\begin{lemma}~\label{bijection} $\mu_\varphi$ is an equilibrium state for  $(f, \varphi)$ if, and only if, $\mu_\phi$ given by $\pi_*\mu_{\phi}=\mu_{\varphi}$ is an equilibrium state for $(F,  {\phi}).$
\end{lemma}

\begin{proof} The uniform contraction of $F$ in $\pi^{-1}({x})$ gives for every $x\in M$  that
$$h_{top}(F, \pi^{-1}({x}))=0.$$ 
Recalling that $P_{f}(\varphi)=P_{F}(\bar\phi)$ by (\ref{mesma pressao}) of Lemma~\ref{le.fizinho}, we apply Ledrappier-Walter's Formula to conclude that
\begin{eqnarray*}
P_{f}(\varphi)&=&P_{F}(\bar\phi)\\
&=&\sup_{\tilde{\eta}}\left\{h_{\tilde{\eta}}(F)+\int{\bar{\phi}} \,d\tilde{\eta}\right\}\\
&=&\!\!\!\!\sup_{\eta;\pi_{\ast}\tilde{\eta}=\eta}\left\{h_{\eta}(f)+\int{\!h_{top}(F, \pi^{-1}(x))\,d\eta(x)}+\int{\varphi}\, d\eta\right\}\\\\
&\leq&P_{f}(\varphi).
\end{eqnarray*}
Thus we must have 
$$h_{ {\mu}}(f)+\int{\varphi} d \mu= h_{{\tilde\mu}}(F)+\int{\bar\phi}d\tilde\mu.$$
for any $\mu\in\mathbb P_f(M)$ and $\tilde\mu\in\mathbb P_F(N)$ such that $\pi_*\tilde\mu=\mu$. 
In partiular, this shows that $\mu_{\varphi}$  is an equilibrium state for $(f, \varphi)$ if and only if $\mu_{\bar\phi}$ given by $\pi_*\mu_{\bar\phi}=\mu_{\phi}$ is  an equilibrium state for $(F, {\bar\phi})$. Since $(F, {\phi})$ and $(F, \bar{\phi})$ have the same equilibrium state, we conclude the proof of the lemma.
\end{proof}

The uniqueness of the equilibrium states for the base dynamics $(f, \varphi)$
was obtained in Theorem~\ref{ee base}. Combining this fact with Lemma~\ref{bijection} we conclude
that there exists only one ergodic equilibrium state for $(F, \phi)$. Since the ergodic Gibbs measure $\mu_{\phi}$ of Theorem~\ref{theo.Gibbs} is an equilibrium state for $(F, \phi)$, we complete the proof of Theorem~\ref{theo finitude}.

\section{Large Deviation Results}\label{LD} 
In this section, we prove Theorem~\ref{th ld1} and Theorem~\ref{th ld2}, which provide large deviation bounds for the equilibrium states $\mu_{\varphi}$ and $\mu_{\phi}$ of $(f, \varphi)$ and $(F, \phi)$ respectively.  Our approach follows the methods developed in \cite{CFV19, Var12, VZ0, You90} to derive large deviation estimates for weak Gibbs measures.

\subsection{Upper bound}
Let $\psi\in C^0(M)$ be a continuous observable and $c\in \mathbb{R}$. For $\beta>0$, $0<\vep<\delta_0/2$ arbitrarily small (where $\delta_0$ is given by~(\ref{estrela})) and $n \ge 1$ fixed, we will estimate the $\mu_{\varphi}$-measure of the set 
$$B_n=\left\{x \in M: \frac{1}{n}\sum_{j=0}^{n-1}\psi(f^j(x))\geq c\right\}.$$
Since $\cG$ is a $\mu_{\varphi}$-full measure set and $\mu_{\varphi}$ satisfies Corollary~\ref{gibbs_n}, we write
% \begin{eqnarray*}\label{eq:decompB}
% B_n\!\!\! &\subset&\!\!\! \{x\in \cG: K_n(x, \vep)>e^{\beta n}\}
% \cup \left(B_n \cap \{x\in \cG: K_n(x, \vep)\le e^{\beta n}\}\right)\nonumber\\
% \!\!\!&\subset&\!\!\!  \{x\in \cG: n_{i+1}-n_i>\frac{\beta n}{2\alpha}\}  \cup \{x\in B_n \cap \cG : K_n(x, \vep)\le e^{\beta n}\}.\label{eq estimative gibbs time}
% \end{eqnarray*} 
\begin{eqnarray*}\label{eq:decompB}
B_n\cap  \cG\!\!\!  &\subset&\!\!\! \{x\in \cG: Ke^{\alpha(n_{i+1}-n_i(x))}>e^{\beta n}\}
\cup\{x\in\cG: Ke^{\alpha(n_{i+1}-n_i(x))}\le e^{\beta n}\}\nonumber\\
&\subset& \{x\in \cG: n_{i+1}-n_i(x)>\frac{\beta n}{2\alpha}\}  \cup \{x\in \cG: Ke^{\alpha(n_{i+1}-n_i(x))}\le e^{\beta n}\},
\end{eqnarray*} 
where $n_i, n_{i+1}$ are consecutive Gibbs times of $x\in\mathcal{G}$ such that $n_{i}\leq n\leq n_{i+1}$. Notice that, in particular, if $E_n\subset \{x\in B_n \cap \cG: Ke^{\alpha(n_{i+1}-n_i(x))}\le e^{\beta n}\}$ is a maximal $(n,\vep)$-separated set then $\{x\in B_n \cap \cG:Ke^{\alpha(n_{i+1}-n_i(x))}\le e^{\beta n}\}$ is
contained in the union of the dynamical balls $B_{2\varepsilon}(x, n)$
centered at points of $E_n$. Therefore, applying Corollary~\ref{gibbs_n} to each $B_{2\varepsilon}(x, n)$ of $E_n$ we obtain
\begin{eqnarray}\label{eq.upper}
    \mu_{\varphi}(B_n)\!\!\! & <& \!\!\! 
        \mu_{\varphi}(\{x:n_{i+1}-n_i(x)>\frac{\beta n}{2\alpha}\})+\mu_{\varphi}(\{x : Ke^{\alpha(n_{i+1}-n_i(x))}\le e^{\beta n}\})\nonumber \\ 
        &<& \mu_{\varphi}(\{x:  n_{i+1}-n_i(x)>\frac{\beta n}{2\alpha}\})+ \sum_{y \in E_n}\mu_{\varphi}(B_{2\varepsilon}(y, n)) \nonumber\\
       &<&  \mu_{\varphi}(\{x:  n_{i+1}-n_i(x)>\frac{\beta n}{2\alpha}\}) + e^{\beta n} \sum_{x \in E_n} e^{S_n\varphi(x)-nP_{f}(\varphi)}
\end{eqnarray}
for every $n$. In the following, we estimate the second term of the last sum. Consider $\si_n$ and
$\eta_n$ the probability measures defined by
$$
\si_n
   =\frac{1}{Z_n} \sum_{x \in E_n} e^{S_n\varphi(x)-nP_{\varphi}(f)} \de_x
   \quad\text{and}\quad
\eta_n
    = \frac{1}{n} \sum_{j=0}^{n-1} f^j_* \si_n,
$$
where $Z_n=\sum_{x \in E_n} e^{S_n\varphi(x)-nP_{\varphi}(f)}.$ Let $\eta\in\mathcal{M}(f)$ be a $f$-invariant measure which is a weak* accumulation point of the sequence $(\eta_n)_{n}$. Assuming that $\cQ$ is a partition of $M$ with diameter smaller
than $\vep$ with $\eta(\partial\cQ)=0$, then each element of $\cQ^{(n)}$ contains a unique point of $E_n$. Proceeding as in the proof of the variational principle (cf. \cite[pp. 219--221]{Wal00}), we compute the entropy of this partition 
$$H_{\si_n}(\cQ^{(n)}) + \int S_n\varphi(x)  \,d\si_n - nP_{\varphi}(f)
    = \log \Big( \sum_{x\in E_n} e^{S_n\varphi(x)-nP_{\varphi}(f)} \Big)
$$
and we obtain that
\begin{equation}\label{eq.upper2}
\limsup_{n \to \infty}
    \frac{1}{n} \log \Big( \sum_{x\in E_n} e^{S_n\varphi(x)-nP_{\varphi}(f)} \Big)
    %\frac{1}{n} \log \Big\{\sum_{x \in E_n} e^{-S_n\psi(x)}\Big\}=
    \leq h_\eta(f) +\int \varphi(x)\, d\eta - P_{\varphi}(f) + \beta.
\end{equation}
Since $\frac {1}{n} \sum_{j=0}^{n-1} \psi (f^j(x)) \ge c$ for every $x\in E_n$, we have that $\int \psi \, d\eta_n \ge c$ for every $n\ge 1$. Thus, by weak$^*$ convergence, it follows that $\int \psi \, d\eta \ge c$ .
%Observe also that $\int \psi\,d\eta\geq c$ by weak$^*$ convergence
%since $E_n$ is contained in $B_n$ and
%$$
% \int g \,d\eta_n
%        %=\frac{1}{n} \sum_{j=0}^{n-1} \int g\circ f^j d\,\si_j
%        = \frac{1}{n} \sum_{j=0}^{n-1} \frac{1}{Z_n} \sum_{x \in E_n} e^{S_n\phi(x)}.\,g\circ f^j(x)
%        = \frac{1}{n} S_n g(x) \geq c.
%$$
Hence, equations~(\ref{eq.upper}) and~\eqref{eq.upper2} together imply 
\begin{align}
\limsup_{n \to +\infty} \frac{1}{n} \log \mu_{\varphi}(B_n)
   & \leq \max\{{\mathcal E}(\beta) \; ,  I(c)+\beta\} \label{thmA-1}
\end{align}
where
$$
{\mathcal E}(\beta)=\limsup_{n\to+\infty} \frac1n \log \mu_{\varphi}\Big(\{x\in \cG: n_{i+1}-n_i>\frac{1}{2\alpha}\beta n\} \Big)
$$
and 
$$
I(c)= \sup \{ h_\eta(f) + \int \varphi d\eta -P_{\varphi}(f)\}.
$$
where the supremum is over all  invariant probability measures $\eta$ such that $\eta(\cG)=1$ and $\int \psi d \eta>c$.
This proves the upper bound of Theorem \ref{th ld1}.

\subsection{Lower bound over ergodic measures} 
For the proof of the lower bound in the large deviations principle, we will need to relate the metric entropy of an invariant measure to the cardinality of dynamical balls covering a set, we recall the following characterization of Kolmogorov–Sinai entropy due to Katok in \cite{Katok}: for any invariant probability measure $\eta$ and arbitrary $\rho>0$ it holds that  
\begin{eqnarray}\label{eq entropy-Katok}
	h_{\eta}(f)=\lim_{\varepsilon\to 0}\limsup_{n\rightarrow +\infty}\frac{1}{n}\log N(n, \vep,\rho)=
	\lim_{\varepsilon\to 0}\liminf_{n\rightarrow +\infty}\frac{1}{n}\log N(n, \vep,\rho),
\end{eqnarray}
where $N(n ,\varepsilon, \rho)$ is the minimum number of $(n,\vep)$-dynamical balls necessary to cover a set of $\eta$-measure greater than $1-\rho$. We refer the reader to \cite[Theorem I.I]{Katok} for a proof of this result.

\begin{theorem}\label{lemma lower ld ergodic measure}
For $ \psi\in C(M)$, $c\in \R$ and $\mu_\varphi$ the equilibrium state of $(f,\varphi)$, we have that
$$
\liminf_{n\rightarrow +\infty}\frac{1}{n}\log \mu_\varphi\left(\{x \in M: \frac{1}{n}\sum_{j=0}^{n-1}\psi(f^j(x))>c\} \right)\ge h_{\eta}(f)+\!\int\varphi d\eta -P_f(\varphi),
$$
for any ergodic probability measure $\eta$ such that $\eta(\cG)=1$ and $\int \psi d \eta>c$.
\end{theorem}
\begin{proof}
Consider $\psi\in C(M)$, $c\in \R$ and $\mu_\varphi$ the equilibrium state of $(f,\varphi)$. As before, denote by $B_n$ the set of points $x\in M$ satisfying $\sum_{j=0}^{n-1}\psi(f^j(x))>cn.$  Let $\eta$ be an ergodic measure such that $\eta(\cG)=1$ and $\int \psi d \eta>c$.  Given  $\beta>0$ small and $\delta=\frac{1}{2}(\int \psi d\eta -c)$, let $\xi>0$ such that $|\psi(x)-\psi(y)|<\delta$ for any points $x,y\in M$ with $d(x,y)<\xi$.

By ergodicity of $\eta$,  if $n_0$ is sufficiently large and $0<\vep<\min\{\xi,\delta_0\}$ is small then the set $D$ of points $x\in \cG$ satisfying 
$$ 
\frac{1}{n}\sum_{j=0}^{n-1}\psi(f^j(x))>c+\delta\,,\,\,\;\; \frac{1}{n}\sum_{j=0}^{n-1}\varphi(f^j(x))< \int\varphi d\eta + \beta
$$
and 
$$ n_{i+1}(x)-n_i(x)\le\frac{1}{2\alpha}\beta n,
$$
for every $n\geq n_0$ has $\eta$-measure at least $1/2$. 

Let $N(n,2\vep,\frac{1}{2})$ be the minimal number  of $(n, 2\vep)$-dynamical balls necessary to cover a set of $\eta$-measure at least $1/2$. From equation~(\ref{eq entropy-Katok}) we have 
 $$N(n,2\vep,\frac{1}{2})\ge e^{(h_{\eta}(f)-\beta)n}$$ 
 for every $n\ge n_0$. 
 % This is possible by the ergodicity of the measure $\eta$ together with the ergodic theorem and equation~(\ref{eq.upper}). 
Since for any maximal $(n,\vep)$-separated set $E_n\subset D$, the dynamical balls $B_{\vep}(x,n)$ centered at points in $E_n$ are pairwise disjoint and  satisfies $D\subset \bigcup_{x\in E_n}B_{\vep}(x,n)\subset B_n$, it follows that 
 $$
\mu_{\varphi}(B_n)\ge \sum_{x\in E_n}\mu_\varphi(B_{\vep}(x,n)) \ge e^{(h_{\eta}(f)+\int \varphi d\eta -P_f(\varphi) -3\beta)n}.
$$
 Therefore, we obtain
$$
\liminf_{n\rightarrow+\infty}\frac{1}{n}\log \mu_\varphi\left(\!\!\{x\in M\!:\! \frac{1}{n}\sum_{j=0}^{n-1}\psi(f^j(x)>c\}\!\!\right)\ge h_{\eta}(f)+\int \varphi d\eta-P_f(\varphi)- 3\beta.
$$
for every $\beta>0$. This completes the proof. 
\end{proof}

The last result is the lower bound in Theorem~\ref{th ld1} using ergodic measures. To establish the general lower bound in the large deviations principle, we will use the non-uniform gluing property to connect orbit segments. 

\subsection{Non-uniform gluing property}\label{gluing property}
The notion of non-uniform gluing property is weaker than the classical specification property but still sufficient for our purposes. For further details on the non-uniform gluing property, we refer the reader to~\cite{BV19}.

\begin{definition} \normalfont{
     We say that $(f, \mu)$ satisfy the \emph{non-uniform gluing property} on a full $\mu$-measure set $X\subset M$ if there exists $\delta>0$ such that for every point $x\in X$ and every $0<\vep<\delta$ there exists an integer $p(x, n, \vep)\geq 1$ satisfying
     $$\displaystyle\lim_{\vep\to 0}\limsup_{n\to +\infty}\frac{1}{n}\, p(x, n, \vep)=0$$ 
     and so that the following holds: given points $x_1, \cdots, x_k$ in $ X$ and times $n_1, \cdots, n_k$ there exist \emph{gluing
times} $p_i\le  p(x_i, n_i, \vep)$  and a point $z$  so that $d(f^j(z),f^j(x_1)\le \vep$ for every $0\le j\le n_1$ and
     $$
     d(f^{j+n_1+p_1+\cdots + n_{i-1}+p_{i-1}}(z),f^j(x_i))\le \vep
     $$
for every $2\le i\le k$ and $0\le j\le n_i$.} Alternatively, using dynamical balls, 
$$
\bigcap_{j=1}^k f^{-\sum_{1\le i\le j}(p_i+n_i)}\left(B_\vep(x_j,n_j)\right)\neq \emptyset.
$$
     \end{definition}
Roughly speaking, this property asserts that for a full $\mu$-measure set, it is possible to shadow the prescribed orbit segments, with transition times between segments bounded by a function $p(x,n,\vep)$ that depends on both the point $x$ and $\vep$, and this function grows slower than sublinearly in $n.$

Recall that we are assuming the density of the preorbit $\{f^{-n}(x)\}_{n\ge 0}$ of every point $x\in \cG$.  In what follows, we relate this property of the past orbit density to certain variations of topological transitivity of a dynamical system. This step is important to show that $f$ has the non-uniform gluing property.

\begin{definition} \normalfont{ Let $M$ be a compact metric space and let $f:M \rightarrow M$ be a continuous map. We say that the dynamical system $(M, f)$ is:
\begin{enumerate}
	\item Strongly Transitive (ST) on a subset $X\subset M$ if, for every open set $U\subset M$ such that $U\cap X\neq \emptyset$ we have $X\subset \bigcup_{j=0}^{+\infty}f^j(U).$
	
	\item Very Strongly Transitive (VST) on a subset $X\subset M$ if, for every open set $U\subset M$ such that $U\cap X\neq \emptyset$, there exists a natural number $n\in\mathbb{N}$ such that $X\subset \bigcup_{j=0}^{n}f^j(U)$.
	\end{enumerate}
}
\end{definition}
Notice that, if $f: M\to M$ is an open map, strong transitivity and very strong transitivity are equivalent properties. In fact, given an open set $U\subset M$ such that $U\cap X\neq \emptyset$ we have that $f^j(U)$ is open for every $j\in\mathbb{N}$. Since $M$ is compact, there exist some $n\in\mathbb{N}$ such that $M\subset \bigcup_{j=0}^{n}f^j(U).$  In particular, $X\subset \bigcup_{j=0}^{n}f^j(U).$

In the proposition below, we establish a relationship between the concept of strong transitivity and the density of a preorbit.

\begin{proposition}
	Let $M$ be a compact metric space and let $f:M \rightarrow M$ be a continuous map. The following statements are equivalent:
	
	\begin{enumerate}
		\item $(M,f)$ is strongly transitive on $X\subset M$.
        \item For every open set $U \subset M$ such that $U\cap X\neq \emptyset$  and every point $x \in X$, there exists $n \in \N$
 such that $x \in f^n(U)$.
		\item For every $x\in X$
		 the preorbit $\{f^{-n}(x)\}_{n\ge 0}$ is dense in $M$.
	\end{enumerate}
\end{proposition}
\begin{proof}
The equivalence $(1)\Leftrightarrow (2)$ follows directly from the definition. To prove $(2)\Rightarrow (3)$, suppose the existence of $x\in X$ such that $\{f^{-n}(x)\}_{n\ge 0}$ is not dense in $M$.  Then there exists a nonempty open set $U \subset M$ such that  $f^{-n}(x)\cap U=\emptyset$ for all $n \in \N$; that is, $x \notin  f^n(U)$ for any $n\in \N$, which contradicts the assumption in $(2)$. The converse implication $(3)\Rightarrow (2)$ can be shown by a similar argument.
\end{proof}
	
% 	\begin{proposition}
% 		If  $(M,f)$ is an open map then, for a set $X$:\\
% 	\begin{center}
% 			(VST) on $X$ $\Leftrightarrow$ (ST) on $X$.
% 	\end{center}
% 	\end{proposition}
% \begin{proof}
% Obviously, (VST) on $X$ $\Rightarrow$ (ST) on $X$. Reciprocally, let $U$ be an open on $M$ such that $U\cap X\neq \emptyset$ and $X\subset \bigcup_{i=1}^{\infty}f^{i}(U)$. Then, each $f^j(U)$ is open in $M$ and $M=\bigcup_{j=1}^{\infty} f^j(U)\cup \bigcup_{j\in \N}(V_j)$ where $M\backslash X\subset \bigcup_{j\in \N}V_j$ is an open cover of $M$. The compactness of $M$ implies that there exists $N \in \N$ such that $X\subset \bigcup_{i=1}^{N}f^{i}(U)$.  Hence, the system is very strongly transitive.

% \end{proof}

Finally, we can prove the non-uniform gluing property for $f$ associated to the weak Gibbs measure. For this, we use that $f$ is very strongly transitive on $\cG$ since $f$ is an open map and we are assuming the density of preorbits.

\begin{proposition}\label{gluing}
Let $f:M\rightarrow M$ and $\varphi:M\rightarrow \R$ be as in Theorem~\ref{th ld1}. Let $\mu_\varphi$ be the   weak Gibbs measure of $(f,\varphi)$. Then $(f,\mu_\varphi)$ satisfies the non-uniform gluing property in $\cG$.
\end{proposition}
\begin{proof}
Let $\vep<\delta_0$ be fixed satisfying~(\ref{estrela}) and $n\ge 1$. Consider $x\in \cG$ and consecutive Gibbs times of $x$ such that $n_k(x)\le n\le n_{k+1}(x)$. Notice that
$$B(x,n_{k+1},\vep)\subseteq B(x,n,\vep)\subseteq B(x,n_{k},\vep)$$
and $$f^{j+n_{k+1}}(B(x,n_{k+1},\vep))=f^j(B(f^{n_{k+1}}(x),\vep)),$$ 
for each $j\in \mathbb{N}$, by condition~(\ref{estrela}).

Recalling that $f$ is very strongly transitive on $\cG$, we can find $N(\vep) \in \N$ such that $\cG\subset \bigcup_{j=0}^{N(\vep)}f^j(B(f^{n_{k+1}}(x),\vep))$. Let $\vep_0$ be the Lebesgue number of the cover $\{ B(f^{n_{k+1}}(x),\vep), \cdots, f^{N(\varepsilon)}(B(f^{n_{k+1}}(x),\vep)) \}$. Thus, for $y\in \cG$ and any $0<\gamma<\vep_0$, we have $B(y,\gamma)\subset f^j(B(f^{n_{k+1}}(x),\vep))$, for some $j\le N(\vep)$. Therefore, there exists $z\in B(x,n,\vep)\cap \cG$ so that
$$
f^{j+n_{k+1}-n}(f^n(z))=f^{j+n_{k+1}}(z)\in B(y,\gamma).
$$
Define $p(x,n,\vep):=N(\vep)+n_{k+1}(x)-n$. Then for any points $x_1,\cdots,x_l$ in $\cG$ and positive integers $n_1,\cdots,n_l$  there are $p_i\le p(x_i,n_i,\vep)$ and $z\in \cG$ such that $z\in B(x_1,n_1,\vep)$ and $f^{n_1+p_1+...+n_{i-1}+p_{i-1}}(z) \in B(x_i,n_i,\vep)$ for every $2\le i \le k$. Moreover, 
$$
\lim_{\vep \rightarrow 0}\limsup_{n \to +\infty}\frac{p(x,n,\vep)}{n}\le \lim_{\vep \rightarrow 0}\limsup_{n \to +\infty}\frac{N(\vep)+n_{k+1}(x)-n_{k}(x)}{n_k(x)}=0
$$
because $(n_k(x))_k$ is non-lacunar.
\end{proof}

From the proof of the previous proposition, we observe that, under the assumption that the Gibbs times form a syndetic set, the gluing function $p(x, n, \vep)$ is uniformly bounded for fixed $\vep.$ This corresponds to the gluing property as defined in \cite{BV19}.

\begin{corollary}\label{cor non glui x glui}
Let $0<\vep<\delta_0$ be fixed.   If the sequence of Gibbs times $(n_{k}(x))_k$ is syndetic for every $x\in \cG$, then the gluing function $p(x, n, \vep)$ is uniformly bounded.
\end{corollary}
\begin{proof}
If $(n_{k}(x))_k$ is syndetic in $\cG$, we have the existence of some constant $C>0$ such that $(n_{k+1}(x)-n_k(x))\le C,$ for every $x\in\cG$ and $k\in\mathbb{N}$. Recalling that, from the proof of Proposition~\ref{gluing}, the gluing function $p(x,n,\vep)$ is given by $p(x,n,\vep)=N(\vep)+n_{k+1}(x)-n(x)$, we obtain the desired property.
\end{proof}

We conclude this subsection by noting that the non-uniform gluing property of $(f, \mu_\varphi)$ implies the corresponding property for the attractor $(F, \mu_\phi)$ associated to the weak Gibbs measure given by Theorem~\ref{theo.Gibbs}. Indeed, since $$\pi^{-1}(B_{\vep}^{f}(x, n))\subset  B_{C\tilde{\vep}}^{F}(\hat{x}, n)\cap \pi^{-1}(B_{\vep}^{f}(x, n))$$
for $n$ large enough, we can proceed as follows. Given points $\hat{x}_1, \cdots, \hat{x}_k$ in $\hat{\mathcal{G}}=\pi^{-1}(\mathcal{G})$ and times $n_1, \cdots, n_k$, we project them to points ${x}_1, \cdots, {x}_k$ in ${\mathcal{G}}$ and find some  $z \in \bigcap_{j=1}^k f^{-\sum_{1\le i\le j}(p_i+n_i)}\left(B^f_\vep(x_j,n_j)\right)$. Then,  it is enough to choose
\begin{align*}
  \hat{z} \in \pi^{-1}(z) &\cap \pi^{-1}\left( \bigcap_{j=1}^k f^{-\sum_{1\le i\le j}(p_i+n_i)}\left(B^f_\vep(x_j,n_j)\right)\right) \\
  &\subset \bigcap_{j=1}^k F^{-\sum_{1\le i\le j}(p_i+n_i)}\left(B^F_{C\tilde\vep}(\hat{x}_j,n_j)\right). 
\end{align*}

\subsection{Lower bound over all invariant measures}
We now turn to the proof of the lower bound in the large deviations principle. At this stage, we have already shown that the lower bound in Theorem~\ref{th ld1} holds when the supremum is restricted to ergodic measures.
To prove the general lower bound, we will use the gluing property of \cite{BV19} and the following lemma, which states that every invariant measure can be approximated by a finite collection of ergodic measures supported in $\cG$. The reader can be find the proof of this lemma in \cite[Lemma 4.3]{Var12}.

\begin{lemma}\label{lemma dec measure}
Let $\eta=\int \eta_xd\eta(x)$ be the ergodic decomposition of a $f$-invariant probability measure $\eta$ such that $\eta(\cG)=1$. Given $\beta>0$ and a finite set $(\varphi_j)_{1\le j\le r}\subset C(M)$ of continuous functions, there are positive real numbers $(a_i)_{1\le i\le k}$ satisfying $a_i\le 1$ and $\sum a_i=1$ and finitely many points $x_1,...,x_k$ such that the ergodic measure $\eta_i=\eta_{x_i}$ from the ergodic decomposition satisfy
\begin{enumerate}
\item $\eta_i(\cG)=1$;
\item $h_{\hat{\eta}}(f)\ge h_{\eta}(f)-\beta$;
\item $|\int\varphi_jd\hat{\eta}-\int \varphi_jd\eta|<\beta$ for every $1\le j \le r$, where $\hat{\eta}=\sum_{i=1}^k a_i\eta_i$.
\end{enumerate}
\end{lemma}

\begin{proof}[Proof of the lower bound in Theorem~\ref{th ld1}] 
Let $\psi \in C(M)$, $c\in \R$ and $\eta$ be any invariant probability measure such that $\eta(\cG)=1$ and $\int \psi d\eta>c$. As before, denote by $B_n$ the set of points $x\in M$ satisfying $\sum_{j=0}^{n-1}\psi(f^i(x))>cn.$  Given $\beta>0$ small enough and $\delta_2=\frac{1}{5}(\int \psi d\eta-c),$ consider the measure $\hat{\eta}=\sum_{j=1}^ka_j\eta_j$, given by Lemma~\ref{lemma dec measure}, that satisfies 
$$
h_{\hat{\eta}}(f)\ge h_{\eta}(f)-\beta,\;\; \int \varphi d\hat{\eta}\geq \int \varphi d\eta+\beta\;\; \mbox{and}\;\; \int \psi d\hat{\eta}\ge\int \psi d\eta - \beta>c+4 \delta_2.
$$
As in the proof of Theorem~\ref{lemma lower ld ergodic measure}, for each $1\le j\le k$,  if $n_0$ is large enough and $0< \vep \le \min\{\xi,\delta_0\}$, the set $D_j$ of points $x\in \cG$ such that 
$$
\frac{1}{n}\sum_{i=0}^{n-1}\psi(f^{i}(x))>\int \psi d\eta_j -\beta\,,\,\, \frac{1}{n}\sum_{i=0}^{n-1}\varphi(f^{i}(x)) <\int \varphi d\eta_j - \beta$$
and
$$ n_{j+1}(x)-n_j(x)\le\frac{1}{2\alpha}\beta n,
$$
for all $n\ge n_0$ has $\eta_j$-measure at least $1/2$.
Moreover, if $E_n^j\subset D_j$  is a $([a_jn],\vep)$-separated set, its cardinality is bounded by $\#E_n^i\ge e^{(h_{\eta_i}f- \beta)[a_in]}$.

% for every $x\in E_n^i$ it follows
%   $$
%     \frac{1}{[a_jn]}\sum_{i=0}^{[a_jn]-1}\psi(x)> \int \psi d\eta_j -\beta\;\; \mbox{and}\;\;\; \left|\frac{1}{[a_jn]}\sum_{i=0}^{[a_jn]-1}\varphi(x)- \int \varphi d\eta_j \right|<\beta.
%     $$

%Recall that $n_i(x)\le n\le n_{i+1}(x)$.

%Take $n_0$ large and $\delta$ small. 
By Proposition~\ref{gluing}, $(f, \mu_\varphi)$ has the non-uniform gluing property in $\cG.$ Furthermore, supposing that the sequence 
$(n_{k}(x))_k$ is syndetic in $\cG$, meaning the existence of $C>0$ such that $n_{k+1}(x)-n_{k}(x)\leq C$ for every $x\in \cG$ and $k\in\mathbb{N}$, it follows from Corollary~\ref{cor non glui x glui} that the gluing function 
$p(x, n, \vep)$ is uniformly bounded by $\tilde{N}(\vep)$ for each fixed $\vep>0.$
We will now use the gluing property to connect orbits of points in the set $E_n^j$. 

Given any sequence $(x_1,\cdots,x_k)$ with $x_j \in E_n^j\subset \cG$, there exists $y\in \cG$  that $\vep/4$-shadows each $x_j$ during $l_j:=[a_jn]$ iterates with a time for switching over from one piece of orbit to the other given by $p_1,\cdots, p_k\le \tilde{N}(\vep/4)$ iterates, respectively.
%( we emphasize that the transition times $p_i=p_i(y)$ are functions of the  underlying points and times). 
Denote by $p_j:=\max_{x\in E^{j}_n}p(x, n, \varepsilon)$ and let $Y_{\tilde {n}}$ be the set of all points $y\in \cG$ obtained as above, where $\tilde{n}=\sum_{j=1}^k(l_j+p_j)$. Notice that, for $\beta>0$ sufficiently small, by uniform continuity of $\psi$, we have that
\begin{align*}
\sum_{j=0}^{\tilde{n}-1}\psi(f^j(y))&\ge \sum_{j=1}^k\left(\sum_{j=0}^{l_j-1}\psi(f^j(x_j))-\delta_2\cdot l_j-\sup|\psi|\cdot p_j\right) \\
    &\ge \sum_{j=1}^k(\int \psi d\eta_j-\delta_2-2\beta)l_j>(\int \psi d\eta-2 \delta_2)n\\
    &>(c+2\delta_2)n+(c+\delta_2)\frac{\beta n}{\sup|\psi|}> (c+\delta_2)\tilde{n}
\end{align*}
with $n$ sufficiently large. This implies that $Y_{\tilde{n}}\subset B_{\tilde{n}}$. 

%Observe also that distinct sequences $(x_1,\cdots,x_k)$ with $x_j \in E_n^j$, give rise to distinct points $y\in\cG$ that shadow the corresponding sequence. Consequently, there are at least $\#E_{n}^1 \cdot ... \cdot \# E_{n}^k$ pairwise disjoint dynamical balls centered at points $y\in Y_{\tilde{n}}$ contained in $B_{\tilde{n}}$. 

 For the sake of completeness, we emphasize that the transition times $p_j$ depend on the collection $(x_1, \ldots, x_k)$ and are not constant. We claim that there exists a subset 
$\tilde{Y}_{\tilde{n}}\subset Y_{\tilde{n}}$ with cardinality bounded from above by $\tilde{N}(\vep/4)^{-k}\cdot \prod_{i=1}^k e^{[a_i \tilde{n}](h_{\eta_i}(f)-\beta)}$ such that the dynamical balls of the family $\{B_{\vep/4}(y,\tilde{n})\}_{y\in \tilde{Y}_{\tilde{n}}}$ are disjoint subsets of $B_{\tilde{n}}$. Indeed, by the pigeonhole principle, for every $1 \le i \le k$ there
 exists $0 \le t_i \le \tilde{N}(\vep/4)$   so that the set
 $$
 \tilde{Y}_{\tilde{n}}:=\{y\in Y_{\tilde{n}}: p_i(y)=t_i\;\mbox{for every}\; 1\le i\le k \}
 $$
has cardinality at least  $ \tilde{N}(\vep/4)^{-k}\, \cdot \#E_{\tilde{n}}^1\cdot \cdots \#E_{\tilde{n}}^k$.
Moreover, notice that the family $\{B_{\vep/4}(y,\tilde{n})\}_{y\in Y_{\tilde{n}}}$ is disjoint:  for $y_1\neq y_2 \in \tilde{Y}_{\tilde{n}}$, there are $1\le i \le k$ and points $x_1\neq x_2 \in E_{\tilde{n}}^i$ such that
$$
d(f^{\sum_{j=0}^{i}l_j+t_j}(y_1),x_1)<\vep/4  \;\; \mbox{and}\;\; d(f^{\sum_{j=0}^{i}l_j+t_j}(y_2),x_2)<\vep/4.
$$
Since $E_{n}^i$ are  $([a_i n],\vep)$-separated sets for $D_i$, there exists $s_i \in [0,[a_i n]]$ such that $d(f^{s_i}(x_1),f^{s_i}(x_2))>\vep$. Thus, for
$\tilde{l}_i=\sum_{j=0}^{i}l_j+t_j$, we have
\begin{eqnarray*}
    d(f^{s_i+\tilde{l}_i}(y_1),f^{s_i+\tilde{l}_i}(y_2))
    &\ge&\\ d(f^{s_i}(x_1),f^{s_i}(x_2))
- d(f^{s_i+\tilde{l}_i}(y_1),f^{s_i}(x_1))
&-& d(f^{s_i+\tilde{l}_i}(y_2),f^{s_i}(x_2))
>\vep/4.
\end{eqnarray*}

Finally, recalling that $\tilde{n}_{i}(y)\leq \tilde{n}\leq \tilde{n}_{i+1}(y)$, where $\tilde{n}_{i}(y)$ are the Gibbs times of $y\in \tilde Y_{\tilde{n}}$ and $({n}_{i}(y))_i$ is syndetic, we estimate the $\mu_\varphi$-measure of the set $B_{\tilde{n}}$ as follows 
 \begin{align*}
    \mu_\varphi(B_{\tilde{n}}) &\ge \sum_{y\in \tilde Y_{\tilde{n}}}        
\mu_\varphi(B_{\vep/4}(y,\tilde{n})) \ge \sum_{y\in \tilde Y_{\tilde{n}}}  K e^{-\alpha((\tilde{n}_{i+1}-\tilde{n}_i)(y))+\sum_{j=0}^{\tilde{n}-1}\varphi(f^j(y)) - \tilde{n}P_{f}(\varphi)}\\
 & \ge 
 K e^{-\alpha C}\#\tilde Y_{\tilde{n}}\,e^{(\int \varphi d\hat{\eta}-\beta - P_{f}(\varphi))\tilde{n}}\\
 &\ge K e^{-\alpha C} \cdot \frac{1}{\tilde{N}(\vep/4)^{k}}\cdot \#E_{n}^1 \cdot \cdots \cdot \# E_{n}^k \,e^{(\int \varphi d\hat{\eta}-\beta - P_{f}(\varphi))\tilde{n}}\\
& \ge K e^{-\alpha C} \cdot \frac{1}{\tilde{N}(\vep/4)^{k}}\cdot e^{\sum_{i=1}^k(h_{\eta_i}(f)-\beta)[a_i \tilde{n}]}\,e^{(\int \varphi d\hat{\eta}-\beta - P_{f}(\varphi))\tilde{n}}.
\end{align*}
Since $$
h_{\hat{\eta}}(f)\ge h_{\eta}(f)-\beta,\;\; \int \varphi d\hat{\eta}\geq \int \varphi d\eta+\beta\;\; ,\;\; \lim_{n\rightarrow +\infty}\frac{p(x, n, \vep)}{n}=0,$$
and $\beta$ is arbitrarily  
small, it follows that
$$ 
\liminf_{n\rightarrow+\infty}\frac{1}{n}\log \mu_\varphi\left(\!\!\{x\in M\!:\! \frac{1}{n}\sum_{j=0}^{n-1}\psi(f^j(x)>c\}\!\!\right)\ge h_{\eta}(f)+\int \varphi d\eta-P_f(\varphi).
$$
This completes the proof of Theorem~\ref{th ld1}.
\end{proof}

We observe that the proof of  Theorem~\ref{th ld2} follows the same general strategy as that of Theorem~\ref{th ld1}. Indeed, by Theorem~\ref{theo finitude}, the equilibrium state $\mu_\phi$
associated to the attractor 
$F$ and the potential 
$\phi$ is a weak Gibbs measure. Furthermore, as noted in Subsection~\ref{gluing property}, the pair $(F, \mu_\phi)$ also satisfies the non-uniform gluing property on the set 
$\hat{\cG}$.

%%%%%%

\end{document}